\documentclass[11pt]{amsart}
\usepackage{geometry}                
\usepackage{graphicx}
\usepackage{amsmath}
\usepackage{amsfonts}
\usepackage{amssymb}
\usepackage{epstopdf}
\usepackage{color}
\usepackage{caption}
\newtheorem{theorem}{Theorem}[section]
\newtheorem{lemma}{Lemma}[section]
\newtheorem{remark}{Remark}[section]
\newtheorem{definition}{Definition}[section]
\newtheorem{proposition}{Proposition}[section]
\newtheorem{corollary}[theorem]{Corollary}

\newcommand{\bv}{\mathbf{v}}

\newcommand{\ba}{\mathbf{a}}
\newcommand{\be}{\mathbf{e}}
\newcommand{\w}{\omega}
\newcommand{\W}{\Omega}
\newcommand{\cC}{\mathcal{C}}
\newcommand{\eps}{\varepsilon}
\newcommand{\uNe}{\underline{N}^{\epsilon}}

\title{Intrinsic Complexity And Scaling Laws: \\ From Random Fields to Random Vectors }
\author{Jennifer Bryson}
\address{Department of Mathematics, University of California, Irvine, CA 92697.}
\email{jabryson@math.uci.edu} 
\author{Hongkai Zhao}
\address{Department of Mathematics, University of California, Irvine, CA 92697.}
\email{zhao@math.uci.edu}
\thanks{The work of Hongkai Zhao is partially supported by NSF grant DMS-1418422 and DMS-1622490.} 
\author{Yimin Zhong}
\address{Department of Mathematics, University of California, Irvine, CA 92697.}
\email{yiminz@math.uci.edu}

\begin{document}

\begin{abstract}
Random fields are commonly used for modeling of spatially (or timely) dependent stochastic processes.  In this study, we provide a characterization of the intrinsic complexity of a random field in terms of its second order statistics, e.g., the covariance function,  based on the Karhumen-Lo\'{e}ve expansion. We then show scaling laws for the intrinsic complexity of a random field in terms of the correlation length as it goes to 0. 
In the discrete setting, it becomes approximate embeddings of a set of random vectors. We provide a precise scaling law when the random vectors have independent and identically distributed entires using random matrix theory as well as when the random vectors has a specific covariance structure.
\end{abstract}
\maketitle
\section{Introduction}
In this study, we denote $a(x,\w)$ a random field, where $\w$ belongs to a probability space $(\W, \Sigma,P)$ and $x$ belongs to a compact domain $D\subseteq R^d$. In other words, we can view the $P$-measureable map: $a(\cdot,\w): \W\rightarrow L^{\infty}(D)$, a set of random functions on $D$ parametrized by $\w \in \W $, or view $a(x,\cdot)$ a set of random variables in $L^2(\W, dP)$ parametrized by the spatial position $x\in D$. 

A useful separable finite approximation for $a(x,\w)$ in practice is an expansion by separating deterministic and stochastic variables of the form
\begin{equation}
\label{eq:1}
a_N(x,\w)=\sum_{n=1}^{N}\phi_n(x)Y_n(\w),
\end{equation}
where $\phi_n(x)$ ($Y_n(w)$) are some functions (random variables), e.g., basis in $L^2(D))$ ($L^2(\W,dP)$). A basic but important question for this separable approximation is: what is the minimum number of terms needed in the expansion (by choosing the appropriate $\phi_n, Y_n$) in order to approximate the true random field $a(x,\w)$ to a given tolerance. From another point of view, if we regard $a(x,\w)$ as a set of random variables (functions) parametrized by $x\in D$ ($\w\in \W$) , the question becomes: what is the least dimension of a linear space that can approximate this set of random variables (functions) to a given tolerance. Accordingly, $Y_n(\w)$ ($\phi_n(x)$), forms the basis of the linear space. This is the dual to the definition of Kolmogorov $n$-width
\footnote{Kolmogorov $n$-width of a set $S$ in a normed space $W$ is its distance to the best $n$ dimensional linear subspace $L_n$:
\[
d_n(S, W) :=\inf_{L_n}\sup_{f\in S}\inf_{g\in L_n} \|f-g\|_W, 
\]
} \cite{kolmogoroff1936uber}
which characterizes the intrinsic complexity or information content  for this family of random variables (functions).
In this work, we will address these questions for a random field $a(x,\w)$ based on the Karhumen-Lo\'{e}ve (KL) approximation. In particular, if there is a length scale, i.e., the correlation length, that characterizes the range of interaction among the family of spatially distributed random variables, we show both upper bound and lower bound of a scaling law for the number of terms needed in the KL approximation in terms of the correlation length. In the discrete setting, it can be formulated as approximate embeddings of a set of random vectors. We provide more precise scaling laws for a few special cases.


Here is the outline of the paper. We first introduce the mathematical formulation of the KL expansion in Section \ref{sec:formulation}. Main results and scaling laws for random fields are presented in Section \ref{sec:RF} evidenced by numerical experiments. Discrete formulation as random vectors embedding and results are presented in Section \ref{sec:embedding}.

\section{Mathematical Formulation}
\label{sec:formulation}
We denote the norm $L^2(D\times \W)$ by $\|\cdot\|_2$ and assume $\|a\|_2<\infty$. The mean field, denoted by $E_a(x)$,  and covariance, denoted by $C_a(x,y)\in D\times D$ of $a(x,\w)$ are defined as 
\begin{equation}
\label{eq:mean}
E_a(x)=\int_{\W}a(x,\w)dP(\w), \quad C_a(x,y)=\int_{\W}[a(x,\w)-E_a(x)][a(y,\w)-E_a(y)]dP(\w). 
\end{equation}
$C_a(x,y)$ can be associated with a compact, self-adjoint and non-negative operator $\cC_a: L^2(D)\rightarrow L^2(D)$ via
\begin{equation}
\label{eq:operator}
(\cC_a u)(x)=\int_D C_a(x,y)u(y)dy, \quad \forall u\in L^2 (D)
\end{equation}
Let $(\lambda_n, e_n(x)), n=1, 2, \ldots$ be the sequence of eigen-pairs associated with $\cC_a$, with $\lambda_1\ge \lambda_2 \ge \ldots \ge \lambda_n \rightarrow 0$ as $n\rightarrow \infty$ and $e_n(x)$ forming an orthonormal basis of $L^2(D)$. Then the KL expansion of the random field $a(x,w)$ is
\begin{equation}
\label{eq:KL}
a(x,w)=E_a(x)+ \sum_{n=1}^{\infty}\sqrt{\lambda_n}e_n(x)Y_n(\w),
\end{equation}
where $Y_n(\w)$ are centered at 0, normalized and pairwise uncorrelated random variables satisfying
\begin{equation}
Y_n(\w)=\frac{1}{\sqrt\lambda_n}\int_D (a(x,\w)-E_a(x))e_n(x)dx, \quad E[Y_n]=0, \quad E[Y_mY_n]=\delta_{mn}.
\end{equation}
For simplicity and without loss of generality, we assume the random field $a(x,\w)$ is centered, i.e.,  $E_a(x)=0$, from now on. We have 
\begin{equation}
\label{eq:trace}
\sum_{n=1}^{\infty}\lambda_n=\|a\|^2_2=\int_D C_a(x,x) dx.
\end{equation}
Moreover, the N-term truncated KL expansion is the best N-term separable approximation of $a(x,\w)$ in $L^2(D\times \W)$. Denote $H=L^2(D), S=L^2(\W,dP)$, we have 
\begin{equation}
\label{eq:truncated}
\|a-\sum_{n=1}^{N}\sqrt{\lambda_n}e_n(x)Y_n(\w)\|^2_2=\inf_{V\subset H, \dim V=N} \|a-P_{V\otimes S}a\|^2_2=\sum_{n=N+1}^{\infty}\lambda_n,
\end{equation}
where $P_{V\otimes S}a$ denotes the $L^2$ projection of $a$ in $V\otimes S$.

\begin{remark}
For a centered discrete and finite process, $\ba_n, n=1, 2, \ldots, N$, ($\lambda_n, \be_n$) are the eigen-pairs of the covariance matrix $\Sigma_{mn}=E[\ba_m\ba_n]$.
\end{remark}

Since the truncated KL expansion \eqref{eq:truncated} is the best separable approximation of a random field in the metric $L^2(D\times \W)$, we introduce the following two characterizations of $\lambda_n$. 

\begin{definition}
\label{def:def1}
Given $\epsilon >0$, $\overline{N}^{\epsilon} =\max_{1\le n < \infty} n, ~s.t.~ \sqrt{\lambda_n}\ge \epsilon$.
\end{definition}
Since $ \overline{N}^{\epsilon}$ is the largest $n$ such that $\sqrt{\lambda_n}\ge \epsilon$, it means that adding one more term to the $\overline{N}^{\epsilon}$-term KL expansion will not reduce the approximation error more than $\epsilon$.  Using $\sum_{n=1}^{\overline{N}^{\epsilon}}\sqrt{\lambda_n}e_n(x)Y_n(\w)$ to approximate $a(x,\omega)$ is analogous to an $\epsilon$-rank approximation to a matrix. 

\begin{definition}
\label{def:def2}
 Given $1\ge \epsilon >0$, $\underline{N}^{\epsilon}=\min n, ~s.t.~  \sum_{m=n+1}^{\infty}\lambda_m\le \epsilon^2 \sum_{m=1}^{\infty}\lambda_m$. 
\end{definition}
From \eqref{eq:trace} and \eqref{eq:truncated} we have 
\begin{equation}
\label{eq:quotient}
\frac{\|a-\sum_{n=1}^{\underline{N}^{\epsilon}}\sqrt{\lambda_n}e_n(x)Y_n(\w)\|_2}{\|a\|_2}\le \epsilon.
\end{equation}
Hence $\underline{N}^{\epsilon}$ is the minimal number of terms in the truncated KL expansion to approximate the random field $a(x,\w)$ with a relative root mean square error of $\epsilon$ or to reach ($1-\epsilon^2)$ of the total variance. 
Since $N$-term truncated KL expansion is the best $N$-term separable approximation for a random field $a(x,\w)$ in $L^2(D\times\Omega)$, 
$\underline{N}^{\epsilon}$ is the minimum number of terms needed in a separable approximation to achieve a relative error $\epsilon$. In other words, if $V\subset L^2(D)$ is linear space and $\frac{\|a-P_{V\otimes S}a\|_2}{\|a\|_2}\le \epsilon$ then $\dim V\ge \underline{N}^{\epsilon}$. Actually $span\{e_n, n=1, 2, \ldots \underline{N}^{\epsilon}\}\subset L^2(D)$ is the best linear space of dimension $\underline{N}^{\epsilon}$ to approximates $a(x,\w)$.  Due to the normalization, $\underline{N}^{\epsilon}$ is invariant under a constant scaling of $a(x,\w)$.

\section{Intrinsic Complexity of a Random Field and Scaling Laws.}
\label{sec:RF}
\subsection{A General Lower Bound and its Scaling Law}
In this Section, we show a general lower bound for $\underline{N}^{\epsilon}$ in terms of the second order statistics, i.e., the covariance function, for a random field. Then use the correlation length to derive a scaling law in terms of the correlation length. 

Define 
\begin{equation}
\label{eq:square}
C_a^T C_a(x,y)=\int_D C_a(x,z) C_a(z,y) dz, 
\end{equation}
which can be associated with the following compact, self-adjoint and non-negative operator
\begin{equation}
( \cC_a^T\cC_a u)(x)= \iint_{D\times D} C_a(x,z) C_a(y,z) u(y) dydz, \quad \forall u\in L^2(D)
\end{equation}
whose eigen-pairs are $(\lambda_n^2, e_n(x)), n=1, 2, \ldots$. Also we have
\begin{equation}
\sum_{n=1}^{\infty}\lambda_n^2=\int_D C_a^T C_a(x,x)dx=\iint_{D\times D} C_a^2(x,y)dxdy.
\end{equation}

\begin{theorem}
\label{th:lb}
\begin{equation}
\label{eq:lb}
 \underline{N}^{\epsilon} \ge (1-\epsilon^2)^2\frac{\left(\int_D C_a(x,x) dx\right)^2}{\iint_{D\times D} C_a^2(x,y)dxdy} =(1-\epsilon^2)^2\frac{\|a\|_2^4}{\iint_{D\times D} C_a^2(x,y)dxdy}
\end{equation}

\begin{equation}
\overline{N}^{\epsilon}\le \epsilon^{-4}\iint_{D\times D} C_a^2(x,y)dxdy.
\end{equation}

\end{theorem}

\begin{proof}

Since $\overline{N}^{\epsilon}=\max  n, ~s.t. ~\sqrt{\lambda_n} \ge \epsilon$ and $ \underline{N}^{\epsilon}=\min n, ~s.t.~ \sum_{m=n+1}^{\infty}\lambda_m \le \epsilon^2 \sum_{m=1}^{\infty}\lambda_m= \epsilon^2\int_D C_a(x,x) dx $, we have
\[
\sum_{m=1}^{ \underline{N}^{\epsilon}}\lambda_m \ge (1- \epsilon^2) \sum_{m=1}^{\infty}\lambda_m= (1-\epsilon^2)\int_D C_a(x,x) dx.
\]
and 
\begin{equation}
\label{eq:CS}
(1-\epsilon^2)^2\left(\int_D C_a(x,x) dx\right)^2\le (\sum_{n=1}^{ \underline{N}^{\epsilon}}\lambda_n )^2\le  \underline{N}^{\epsilon}  (\sum_{n=1}^{\underline{N}^{\epsilon}}\lambda_n^2) \le \underline{N}^{\epsilon}  (\sum_{n=1}^{\infty}\lambda_n^2)=\underline{N}^{\epsilon} \iint_{D\times D} C_a^2(x,y)dxdy,
\end{equation}
and 
\begin{equation}
\label{eq:upper}
\iint_{D\times D} C_a^2(x,y)dxdy = \sum_{n=1}^{\infty}\lambda_n^2 \ge \sum_{n=1}^{\overline{N}^{\epsilon}}\lambda_n^2 \ge \overline{N}^{\epsilon} \epsilon^4.
\end{equation}
Hence we get 
\begin{equation}
\label{eq:lower}
\underline{N}^{\epsilon} \ge (1-\epsilon^2)^2\frac{\left(\int_D C_a(x,x) dx\right)^2}{\iint_{D\times D} C_a^2(x,y)dxdy}
\end{equation}
and 
\begin{equation}
\overline{N}^{\epsilon}\le \epsilon^{-4}\iint_{D\times D} C_a^2(x,y)dxdy.
\end{equation}

\end{proof}

\begin{remark}
For a centered discrete and finite process, $\ba_n, n=1, 2, \ldots, N$, with covariance matrix $\Sigma_{mn}=E[\ba_m\ba_n]$, $C_a^T C_a(x,y)$ becomes $\Sigma^T\Sigma$ and $\iint_{D\times D} C_a^2(x,y)dxdy$ becomes $trace(\Sigma^T\Sigma)$.
\end{remark}

If the random field $a(x,\w)$ is stationary and the covariance function is of the form $C_a (x,y)=f(\frac{x-y}{\sigma})$, 
e.g., a Gaussian process with the Gaussian covariance kernel $C_a(x,y)=\exp(-\frac{|x-y|^2}{\sigma^2})$,
 where $\sigma$ introduces a length scale, 
 i.e., the correlation length, we show a scaling law for the lower bound for $\underline{N}^{\epsilon}$ in Theorem \ref{th:lb} in term of $\sigma$. In other words, with the correlation length $\sigma$, the intrinsic degrees of freedom for the random field with spatial variable defined in a $d$-dimensional bounded domain is at least of order $O(\sigma^{-d})$. We will provide upper bounds and their scaling laws based on regularity/smoothness of the covariance function. In particular, if the covariance function is analytic, the scaling law for both the lower bound and the upper bound is sharp. Later we will  
 present numerical experiments to verify the sharpness of our estimates for a few popular models for random fields. 

 \begin{theorem}
 \label{th:l-scaling}
 For a stationary random field $a(x,\w), ~x\in D\subset \mathbb{R}^d$, $D$ compact, with covariance function $C_a (x,y)=f(\frac{x-y}{\sigma})$, and $\int_{\mathbb{R}^d}f^2(x)dx=I<\infty$, then $\exists c(D, f, \epsilon) >0$ such that
 \begin{equation}
 \label{eq:l-scaling}
 \underline{N}^{\epsilon} \ge c(D,f, \epsilon)\sigma^{-d}, \quad \mbox{as } \sigma \rightarrow 0.
 \end{equation}
 \end{theorem}
 \begin{proof}
 Let $S=D+D$ and $|D|$, $|S|$ denote the volume of $D$ and $S$ respectively. 
\begin{equation}
\iint_{D\times D} C_a^2(x,y)dxdy=\iint_{D\times D} f^2(\frac{x-y}{\sigma}) dxdy \underset{\eta=x+y}{\overset{\xi=x-y}{=\joinrel=}}
\frac{1}{2^d}\iint_{D_{\eta,\xi}}f^2(\frac{\xi}{\sigma})d\xi d\eta \le \frac{I|S|}{2^d}\sigma^{d}.
\end{equation}
 From \eqref{eq:lb}, one has 
 \begin{equation}
  \underline{N}^{\epsilon} \ge (1-\epsilon^2)^2 \frac{2^d f^2(0)|D|^2}{I|S|}\sigma^{-d}.
 \end{equation}

 \end{proof}

\begin{remark}
The sharpness of the lower bound estimate for $\underline{N}^{\epsilon}$ depends on the sharpness of the Cauchy-Schwartz inequality used in \eqref{eq:CS}. If the top $O(\sigma^{-d})$ singular values are about the same order, then this lower bound is sharp. Heuristically, by a scaling argument, the degrees of freedom of a random field $a(x,\omega)$ parametrized by $x$ in a bounded domain $D\subset R^d$ should grow as $O(\sigma^{-d})$ as $\sigma \rightarrow 0$ since the random field decorrelates at the length scale $\sigma$. In the work \cite{engquist11approximate}, a similar lower bound for the approximate separability of the Green's function of the high frequency wave fields was proved and its scaling law in terms of the wavelength was derived. The physical explanation is that two Green's functions with sources separated by more than a wavelength become decorrelated. The estimate was shown to be sharp in many situations. 
\end{remark}

\subsection{Upper Bound and Scaling Law Using Smoothness}
For a given random field (or a two variable function in general), an upper bound for the number of terms needed in a separable approximation with tolerance $\epsilon$ as $\epsilon\rightarrow 0$ is also very useful. If it grows slowly, e.g., polylog of $\epsilon^{-1}$, it implies the existence of low rank approximation once the random field (or function) is discretized. The upper bound are usually shown by choosing an appropriate basis. For example, polynomial basis and Taylor expansion is often used for highly separable approximation based on analyticity. This property are constantly explored to develop fast computation algorithms for matrix vector multiplication, e.g., fast multipole method \cite{greengard1987fast,greengard1988rapid, rokhlin1990rapid}, butterfly method \cite{candes2009fast, michielssen1996multilevel}, and direct or structured inverse of a matrix, e.g., HSS, H-matrix, \cite{bebendorf2003existence, borm2010approximation,ho2016hierarchical, martinsson2009fast, schmitz2012fast, xia2013efficient}. 
In a recent work \cite{ambikasaran2014fast}, fast direct methods for Gaussian processes were developed using the fact that for the most commonly used
covariance functions, the corresponding discrete covariance matrix
can be hierarchically factored into a product of block low-rank updates of the identity matrix.  In \cite{udell2017nice}  log-rank approximation of a matrices whose entries are piece-wise analytic functions of certain latent variables were shown. Here we will provide upper bounds and scaling laws for the KL expansion of a random field based on 
various smoothness (or regularity) assumptions of the covariance function, $C_a(x,y)$. The result is based on a decay rate estimates of the eigenvalues of the self-adjoint, non-negative operator \eqref{eq:operator}. To make the explanation more self-contained, we will first quote a few results from \cite{schwab2006karhunen}. 

\begin{lemma}[Lemma 2.16 \cite{schwab2006karhunen}]
\label{le:1}
Let $H$ be a Hilbert space and $\mathcal{C}:H\to H$ be a symmetric, non-negative and compact operator whose eigen-pair sequence is $(\lambda_m,\phi_m)_{m\ge 1}$. if $\mathcal{C}_m$ is an operator of rank at most $m$, then 
\begin{equation}
\lambda_{m+1}\le \|\mathcal{C}-\mathcal{C}_m\|_{B(H)}
\end{equation}
where $\|\cdot \|_{B(H)}$ denotes the induced norm of $H$. 
\end{lemma}


Based on the regularity of the covariance function, one can design projection to appropriate finite element spaces composed of piecewise polynomial basis to approximate $\mathcal{C}$ by a finite rank operator. 

\begin{proposition}[Proposition 2.17  \cite{schwab2006karhunen}]
\label{prop:proj}
Let $S_h^p$ denote the space of discontinuous, piecewise polynomial functions of degree less than $p>0$ on a quasi-uniform triangulation $\mathcal{T}_h$ of $D$ with mesh size $h>0$ and denote by $n=dim S_h^p=O(h^{-d})$ its dimension. Denote by $\mathcal{P}_h: L^2(D)\rightarrow S_h^p(D)$ the $L^2(D)$ projection. As $h/p\rightarrow 0$, for any $f \in H^p(D)$ it holds
\begin{equation}
\label{eq:Hp}
\|f-\mathcal{P}_hf\|_{L^2(D)}\le C(p)n^{-\frac{p}{d}}
\end{equation}
and, if $f$ is analytic, there are constants $c, C >0$ such that, as $p\rightarrow \infty$ on a fixed triangulation $\mathcal{T}_h$ of $D$,
\begin{equation}
\|f-\mathcal{P}_hf\|_{L^2(D)}\le C\exp(-cn^{\frac{1}{d}}).
\end{equation}
\end{proposition}

The above approximation properties in terms of smoothness condition combined with Lemma \ref{le:1}
leads to estimates of the decay of the eigenvalues of the symmetric, non-negative and compact operator $\mathcal{C}_a$ by constructing a finite rank operator approximation using an appropriate $S_h^p$ space to approximate the covariance function $C_a(x,y)$. The following result is a combination of Proposition 2.18 and 2.21 in  \cite{schwab2006karhunen}.

\begin{proposition}
\label{prop:decay}
Let $\mathcal{C}_a$ be the symmetric, non-negative and compact operator on $L^2(D)$ defined in \eqref{eq:operator} with eigen-pair $(\lambda_n, e_n(x))_{n\ge 1}$, then we have the following estimates involving positive constants (depending on $C_a(x,y)$, $D$) $C_0, C_1, c_1 >0$ such that for any $n$,

\begin{enumerate}
\item
if $C_a(x,y)$ is $H^p$, $0\le \lambda_n \le C_0 n^{-\frac{p}{d}}$.
\item
if $C_a(x,y)$ is analytic, $0\le \lambda_n \le C_1 \exp(-c_1 n^{\frac{1}{d}})$.
\end{enumerate}
\end{proposition}

Below we show corresponding scaling laws for the upper bounds for $\underline{N}^{\epsilon}$, which is the minimal number of terms in the truncated KL expansion to approximate the random field $a(x,\w)$ with a relative r.m.s error of $\epsilon$ (see Definition \ref{def:def2}), in terms of the correlation length in the covariance function.

\begin{theorem}
\label{th:scaling}
For a stationary random field $a(x,\w), ~x\in D\subset \mathbb{R}^d$, $D$ bounded and compact, with covariance function $C_a (x,y)=f(\frac{x-y}{\sigma})$, we have the following upper bounds for  $\underline{N}^{\epsilon}$, as $\sigma \rightarrow 0$,
\begin{enumerate}
\item
if $f \in H^p(\mathbb{R}^d), p > d$: there exists a constant  $C(D, f, d, \epsilon)>0$,  such that 
\[
 \underline{N}^{\epsilon} \le C(D, f, d, \epsilon) \sigma^{-(1+\frac{d}{2(p - d)})d}.
\]
\item
if $f$ is analytic in $\mathbb{R}^d$: there exists a constant  $C(D, f, d, \epsilon)>0$,  such that
\[
\underline{N}^{\epsilon}\le C(D, f, d, \epsilon)\sigma^{-d}|\log \sigma|^d.
\]
\end{enumerate}

\end{theorem}
\begin{proof}
1) Since $C_a (x,y)$ is $H^p$, from the approximation theory in Sobolev space~\cite{brenner2007mathematical}, we have 
\begin{equation}
\label{eq:bound}
0\le \lambda_n \le C(d) n^{-\frac{p}{d}} \|\nabla^p C_a \|_{L^2(D\times D)}
\end{equation}
where $\nabla^p$ denotes all $p$-th order derivatives and $C(d)>0$ is a constant independent of $C_a$. If $p>d$, 
\begin{equation}
	\sum_{n=\underline{N}^{\epsilon}+1}^{\infty} \lambda_n \le C(d) \|\nabla^p C_a\|_{L^2(D\times D)} \sum_{n=\underline{N}^{\epsilon}+1}^{\infty} n^{-\frac{p}{d}} \le C(d)\|\nabla^p C_a\|_{L^2(D\times D)} \left({\underline{N}^{\epsilon}}+1\right)^{1-\frac{p}{d}}.
\end{equation}
for some constant $C(d) > 0$. We then find the upper bound for $\underline{N}^{\epsilon}$ by estimating the smallest $\underline{N}^{\epsilon}$ that
\begin{equation}
C(d)\|\nabla^p C_a\|_{L^2(D\times D)} \left({\underline{N}^{\epsilon}}+1\right)^{1-\frac{p}{d}} \le \epsilon^2\|a\|_2^2 =\epsilon^2f(0)|D|,
\end{equation}
where $|D|$ is the volume of $D$. Therefore there exists constant $C(d) > 0$, which we keep the same notation, such that
\begin{equation}
\label{eq:ep1}
	\underline{N}^{\epsilon} \le C(d) \epsilon^{\frac{2d}{d - p}} (\|a\|_2^2)^{\frac{d}{d - p}} \|\nabla^{p} C_a\|_{L^2(D\times D)}^{\frac{d}{p - d}}.
\end{equation}
Since $C_a (x,y)=f(\frac{x-y}{\sigma})$ and $f(x) \in H^p(\mathbb{R}^d)$, we have
\begin{equation}
\|\nabla^{p} C_a\|_{L^2(D\times D)}^{\frac{d}{p-d}} \le \|f\|_{H^p(\mathbb{R}^d)}|S|\sigma^{-(1+\frac{d}{2(p - d)})d},
\end{equation}
where $S=D+D$ and $|S|$ is the volume of $S$. Combining the above two equations we get the scaling law for the upper bound.
\\ \\
2) Denote the difference set $T = D-D$, without loss of generality, we may assume $T$ contains an open ball $B\subset \mathbb{R}^d$ centered at zero. 

First we show that $c_1 = O(\sigma)$ and $C_1 = O(\sigma^{-d})$ in Proposition~\ref{prop:decay}. Since the covariance function $C_a(x, y) = f(\frac{x-y}{\sigma})$ is stationary then by Bochner's theorem, 
\begin{equation}
f(t) = \int_{\mathbb{R}^d} e^{i \xi\cdot t} d\mu(\xi),\quad t\in\mathbb{R}^d,
\end{equation}
for some positive finite measure $\mu$. Define following non-empty convex set $\Theta$ by the Laplace transform of the measure $\mu$,
\begin{equation}
\Theta = \{ \theta \in \mathbb{R}^d : \int_{\mathbb{R}^d}  e^{\theta\cdot \xi} d\mu(\xi) < \infty \},
\end{equation}
since $f(t)$ is analytic on $\mathbb{R}^d$, then the analyticity of $f(t)$ can be continuously extended to the strip $\mathbb{R}^d + i \Theta\subseteq \mathbb{C}^d$ by Theorem 2.7.1 of~\cite{lehmann2006testing} with
\begin{equation}
f(z) = \int_{\mathbb{R}^d} e^{i \xi\cdot z} d\mu(\xi),\quad z\in \mathbb{R}^d + i \Theta.
\end{equation}
For simplicity, we assume there exists $B_{\rho}(0)$ a closed ball of radius $\rho > 0$ centered at zero contained in $\Theta$ and denote $K = C_R + i B_{\rho}(0)$ where $C_R = [-R, R]^d$ is a hypercube covering $T = D-D$ and select $R' > R$. Then by Theorem 4.2 of~\cite{trefethen2017multivariate} and the Theorem 11 of Chapter 5 in~\cite{bochner1948several} with the contour integral, we have following error estimate
\begin{equation}\label{eq:estimate2}
\|f - \mathcal{P}_h f\|_{L^{\infty}(K)}\le M \sup_{K} |f|\left(\frac{ R'}{ \rho}(1 +  \frac{\rho}{R'})^{-n^{1/d}}\right)^d
\end{equation}
for some constant $M > 0$, where $\mathcal{P}_h$ is the projection to $S_h^p$ defined in Proposition \ref{prop:proj}. After scaling with the correlation length $\sigma$, let $\tilde{f}(z) = f(\frac{z}{\sigma})$ for $z\in K_{\sigma} = C_R + i B_{\sigma\rho}(0)$, and then with~\eqref{eq:estimate2} we can estimate $\lambda_n$ for $C_a(x, y)$ by 
\begin{equation}
\begin{aligned}
0 \le \lambda_n &\le \tilde{M} \sup_{K_{\sigma}} |\tilde{f}|\left( \frac{R'}{\rho}\right)^d \frac{1}{ \sigma^d} \exp\left( -dn^{1/d}\log(1 +  \frac{\sigma\rho}{R'})\right)\\&= \tilde{M} \sup_{K} |f|\left( \frac{R'}{\rho}\right)^d \frac{1}{ \sigma^d} \exp\left( -dn^{1/d}\log(1 +  \frac{\sigma\rho}{R'})\right)
\end{aligned}
\end{equation}
for some constant $\tilde{M} > 0$. Here $\sup_{K_{\sigma}} |\tilde{f}| = \sup_{K} |{f}|$ is due to  $|f|$ attains the supremum at the set $0 + iB_{\rho}(0)$ and $|\tilde{f}|$ attains the supremum on the set $0 + iB_{\sigma\rho}(0)$. Since $\log(1 + \frac{\sigma\rho}{R'} ) = O(\sigma)$ as $\sigma\to 0$, we have $0\le \lambda_n \le C_1 \exp(-c_1 n^{\frac{1}{d}})$ for $c_1 = O(\sigma)$, $C_1 = O(\sigma^{-d})$.
From this estimate, we provide a scaling for the upper bound of $\uNe$.
\begin{equation}
\sum_{n=\uNe + 1}^{\infty} \lambda_n \le C_1 \sum_{n=\uNe + 1}^{\infty} \exp(-c_1 n^{1/d}) \le C_1 \int_{\uNe}^{\infty} \exp(-c_1 x^{1/d}) dx.
\end{equation}
By change of variable $y = x^{1/d}$ and integration by parts, we have
\begin{equation}
\label{eq:estimate1}
C_1\int_{\uNe}^{\infty} \exp(-c_1 x^{1/d}) dx \le C(d)\frac{C_1}{c_1} \exp( -c_1 (\uNe)^{1/d}) \left( (\uNe)^{(d-1)/d} + c_1^{-(d-1)}\right).
\end{equation}
Since we already know that $\uNe \ge O(\sigma^{-d})$ from Theorem \ref{th:l-scaling} and $c_1 = O(\sigma)$, therefore $(\uNe)^{(d-1)/d} \ge O(c_1^{-(d-1)})$. We can absorb the $c_1$ term into $\uNe$ term and change the constant $C(d)$, which we keep the same notation.
Hence 
\[
	\sum_{n=\uNe + 1}^{\infty} \lambda_n\le C(d)\frac{C_1}{c_1}\exp( -c_1 (\uNe)^{1/d})  (\uNe)^{(d-1)/d}.
\]
We find the upper bound for $\uNe$ by estimate the smallest $\uNe > 0$ such that,
\[
C(d)\frac{C_1}{c_1} \exp( -c_1 (\uNe)^{1/d})  (\uNe)^{(d-1)/d} \le \epsilon^2 \|a\|_2^2.
\]
We compute an upper bound by $\uNe \le y^d$, where $y>0$ solves the following equation 
\[
\exp(-c_1 y) y^{d-1} = \frac{c_1\epsilon^2 \|a\|_2^2}{C_1C(d)}
\]
which is the same as
\begin{equation}
\label{eq:5}
y - \frac{d-1}{c_1}\log y = \frac{1}{c_1}\log \left(\frac{C(d)C_1}{c_1\epsilon^2 \|a\|_2^2}\right)
\end{equation}
Denote the right-hand-side as $b$, then 
$
b > \frac{2d}{c_1}
$
if provided $\sigma$ is small enough such that $\frac{c_1}{C_1} < \frac{C(d)}{\exp(2d)\epsilon^2 \|a\|_2^2}$. It is obvious $y\ge b$ from \eqref{eq:5}. Let $y = b(1 + t)$, for some $t \ge 0$, we have the following
\[
b(1+t) -\frac{d-1}{c_1}\log b - \frac{d-1}{c_1}\log(1+t) = b.
\]
and
\[
\left( b - \frac{d-1}{c_1}\right) t \le \frac{d-1}{c_1} \log b
\]
since $\ln(1+t)\le t$ for all $t \ge 0$. From $b > \frac{2d}{c_1}
~\Rightarrow ~ b - \frac{d-1}{c_1} > \frac{b}{2} $.
we obtain
\[
t \le \frac{2(d-1)}{c_1}\frac{\log b}{b},
\]
which means
\[
y \le b\left( 1 + \frac{2(d-1)}{c_1}\frac{\log b}{b}\right) = b + \frac{2(d-1)}{c_1}\log b
\quad \mbox{and} 
\quad \uNe \le \left(b + \frac{2(d-1)}{c_1}\log b\right)^d.
\]
Plug in 
$b = \frac{1}{c_1}\log \left(\frac{C(d)C_1}{c_1\epsilon^2 \|a\|^2}\right)$ and use the fact that $b > \frac{2d}{c_1}$,  we have
$$\log \left(\frac{C(d)C_1}{c_1\epsilon^2 \|a\|^2}\right) \ge 2d > 1.$$
From the fact that $\log x \le x- 1 < x$ for $x\ge 1$,
\[
\log b = \log\left(\frac{1}{c_1}\log \left(\frac{C(d)C_1}{c_1\epsilon^2 \|a\|^2}\right)\right)< \log\left(  \frac{1}{c_1}\left( \frac{C(d)C_1}{c_1\epsilon^2 \|a\|^2}  \right) \right).
\]
 We can show an upper bound for $\uNe$ using $c_1 = O(\sigma)$ and $C_1 = O(\sigma^{-d})$,
\begin{equation}
\label{eq:ep2}
\uNe \le \left(\frac{1}{c_1}\log \left(\frac{C(d)C_1}{c_1\epsilon^2 \|a\|^2}\right) + \frac{2(d-1)}{c_1}\log\left(  \frac{C(d)C_1}{c_1^2\epsilon^2 \|a\|^2}   \right)\right)^d = O\left(\Big|\frac{\log \sigma}{\sigma}\Big|^d \right).
\end{equation}

\end{proof}

\begin{remark}
If $f\in H^p$, the finite element space $S_h^p$ and its approximation property \eqref{eq:Hp} in Proposition \ref{prop:proj} is optimal up to some constant. Hence, we expect the upper bound for $\lambda_n$ and our scaling law for the upper bound of $\uNe$ is also sharp. In other words, the non-smoothness of the random field or its derivatives may introduce extra complexity factor. 
When $f$ is analytic, the correlation length is the only length scale in the random field. So we get sharp scaling laws for both the lower bound and upper bound for the intrinsic complexity.
\end{remark}

From the proof of Theorem \ref{th:scaling}, i.e., equations \eqref{eq:ep1}, \eqref{eq:ep2}, we can deduce the upper bound for $\uNe$ in terms of $\epsilon$.

\begin{theorem}
For a stationary random field $a(x,\w), ~x\in D\subset \mathbb{R}^d$, $D$ bounded and compact, with covariance function $C_a (x,y)=f(x-y)$, we have the following upper bounds for  $\underline{N}^{\epsilon}$, as $\epsilon \rightarrow 0$,
\begin{enumerate}
\item
if $f$ is $H^p$ and $p > d$: $\exists C(D, f, d)>0$ s.t. $\underline{N}^{\epsilon} \le  C(D, f, d) \epsilon^{\frac{2d}{d-p}}$, 
\item
 if $f$ is analytic: $\exists C(D, f, d)>0$ s.t. $\underline{N}^{\epsilon}\le C(D, f, d)|\log \epsilon|^d$.
\end{enumerate}

\end{theorem}

\subsection{Experiements}
\label{sec:test}

Here we provide numerical experiments to verify our estimate for the intrinsic complexity of a random field $a(x,\omega), x\in D\subset R^d$ with a given covariance function $C(x,y)$ through the KL expansion. In particular, we demonstrate the sharp scaling laws $\underline{N}^{\epsilon}=O(\sigma^{-d})$ as $\sigma\rightarrow 0$ for two covariance functions that are often used to model  random fields in practice:
\begin{equation}
\label{eq:kernel}
\tilde{C}_{\sigma}(x,y)=\exp\left(-\frac{|x-y|^2}{\sigma^2}\right),  \quad \quad \hat{C}_{\sigma}(x,y)=\exp\left(-\frac{|x-y|}{\sigma}\right), 
\end{equation}
In our experiments, we take $D$ to be the unit interval, the unit square and the unit sphere (where $|x-y|$ is the geodesic distance) and 
compute the eigenvalues of $\tilde{C}_{\sigma}$ and $\hat{C}_{\sigma}$ that are discretized uniformly in $D$ with a grid size $h$.
We show eigenvalue behavior with different resolution $r$, the number of points per $\sigma$ ($\sigma=rh$), different tolerance $\epsilon$, and the scaling law as $\sigma\rightarrow 0$.

{\bf Example 1: unit interval.}
Figure \ref{fig:1D-ker1} shows the test results for kernel $\tilde{C}_{\sigma}(x,y)=\exp(-\frac{|x-y|^2}{\sigma^2})$.  Figure \ref{fig:1D-ker1} (a) shows how $\underline{N}^{\epsilon}$, the minimal number of terms in truncated KL expansion to approximate the random field $a(x,\w)$ with a relative r.m.s error of $\epsilon$ (Definition \ref{def:def2}), as a function of the discretized resolution $r$, where $h=\sigma/r$, with a fixed $\sigma=0.02$. Figure \ref{fig:1D-ker1} (b) shows how $\overline{N}^{\epsilon}$, the number of eigenvalues that are larger than $\epsilon$ (Definition \ref{def:def1}), as a function of the discretized resolution $r$. As can be seen, due to the normalization of the definition $\underline{N}^{\epsilon}$ (relative error normalized by $\|a\|_2$) and the smoothness of the kernel $\tilde{C}_{\sigma}$, $2$ or $3$ points per $\sigma$ is fine enough to resolve the covariance kernel in numerical computation. Although the dimension of the discretized covariance matrix becomes larger  as the resolution $r$ becomes higher, $\underline{N}^{\epsilon}$ remains the same. In other words, the random field can be approximated with a relative r.m.s error of $\epsilon$ by projection to $V\otimes L^2(\W,dP)$, where $V\subset L^2(D)$ is a linear subspace of dimension $\underline{N}^{\epsilon}$ spanned by the leading eigen-functions $e_n, n=1, 2, \ldots, \underline{N}^{\epsilon}$.
While the unnormalized $\overline{N}^{\epsilon}$, the number of eigenvalues that are larger than $\epsilon$ (or $\epsilon$-rank approximation), grows slowly with the discretization resolution $r$, since the dimension of the corresponding discretized matrix becomes larger as the resolution becomes higher. Hence it takes a higher rank matrix to approximate it uniformly well. 
The growth is likely to be logarithmic since $\tilde{C}$ is analytic and can be approximated by polynomials with an error that decays exponentially. 
Figure \ref{fig:1D-ker1} (c)\&(d) shows $\underline{N}^{\epsilon}$ and $\overline{N}^{\epsilon}$ as a function of $\sigma^{-1}$ for different $\epsilon$. We see a clear linear growth for both. The numerical result is based on eigenvalue decomposition of the matrix that corresponds to a discretization of  $\tilde{C}_{\sigma}(x,y), x, y \in [0,1]$ with $h=0.25\sigma$. 

\begin{figure}[!htb]
	\begin{center}
		\hspace{-0.6cm}
		\begin{minipage}{0.46\textwidth}
			\includegraphics[width=1\textwidth]{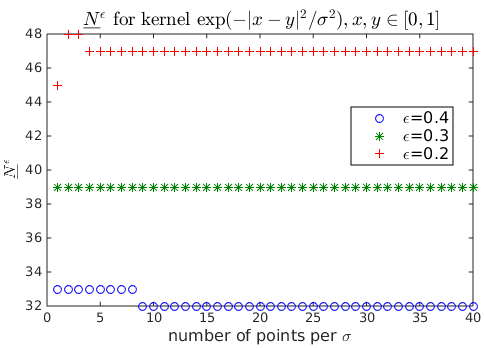}
			\caption*{(a) $\underline{N}^{\epsilon}$ vs. $r$ with fixed $\sigma = 0.02$}
		\end{minipage}
 \hfil
		\begin{minipage}{0.46\textwidth}
		\includegraphics[width=1\textwidth]{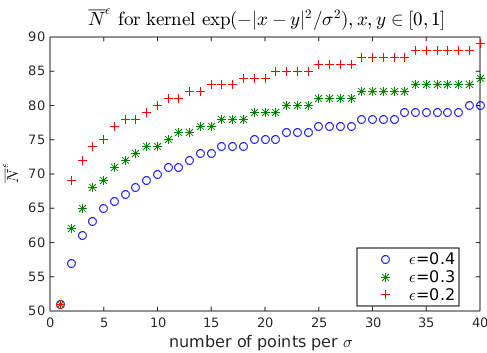}
		\caption*{(b) $\overline{N}^{\epsilon}$ vs. $r$ with fixed $\sigma = 0.02$}
		\end{minipage}
		\\
		\vspace{0.3cm}
		\hspace{-0.6cm}
		\begin{minipage}{0.46\textwidth}
		\includegraphics[width=1\textwidth]{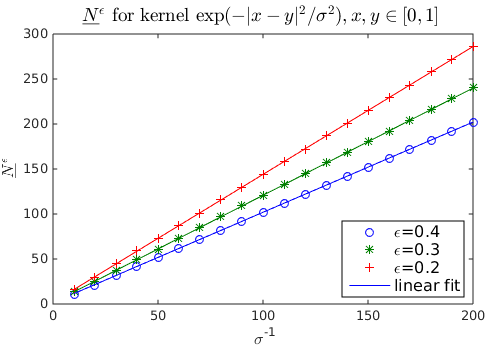}
		\caption*{(c) $\underline{N}^{\epsilon}$ vs. $\sigma$ with fixed $h=0.25\sigma$}
		\end{minipage}
		\hfil
		\begin{minipage}{0.46\textwidth}
		\includegraphics[width=1\textwidth]{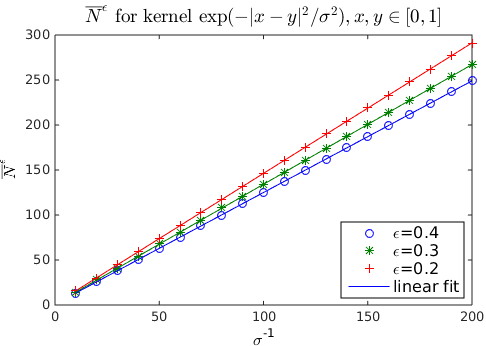}
		\caption*{(d) $\overline{N}^{\epsilon}$ vs $\sigma$ with fixed $h=0.25\sigma$}
		\end{minipage}
		\caption{Eigenvalue behavior for kernel function $\tilde{C}_{\sigma}=\exp(-\frac{|x-y|^2}{\sigma^2})$ on the unit interval.}
		\label{fig:1D-ker1}
	\end{center}
\end{figure}


Figure \ref{fig:1D-ker2} shows the corresponding results for kernel $\hat{C}_{\sigma}(x,y)=\exp(-\frac{|x-y|}{\sigma})$. Figure \ref{fig:1D-ker2} (a) shows $\underline{N}^{\epsilon}$ as a function of the discretized resolution $r$ with a fixed $\sigma=0.02$. Due to the singularity of $\hat{C}_{\sigma}(x,y)$ at $x=y$, $\underline{N}^{\epsilon}$ increases as the singularity is more resolved. However, it appears to be plateaued eventually.  While Figure \ref{fig:1D-ker2}(b) shows that the unnormalized $\overline{N}^{\epsilon}$, the number of eigenvalues that are larger than $\epsilon$, grows with the discretization resolution $r$ and with a rate faster than that of the smooth kernel $\tilde{C}$. 
Figure \ref{fig:1D-ker2} (c)\&(d) shows $\underline{N}^{\epsilon}$ and $\overline{N}^{\epsilon}$ as a function of $\sigma^{-1}$ for different $\epsilon$. Again, we see a clear linear growth. However,  both $\underline{N}^{\epsilon}$ and $\overline{N}^{\epsilon}$ grow with a faster rate than that of the analytic covariance function $\tilde{C}_{\sigma}(x,y)$.

\begin{figure}[!htb]
	\begin{center}
		\hspace{-0.6cm}
		\begin{minipage}{0.46\textwidth}
			\includegraphics[width=1\textwidth]{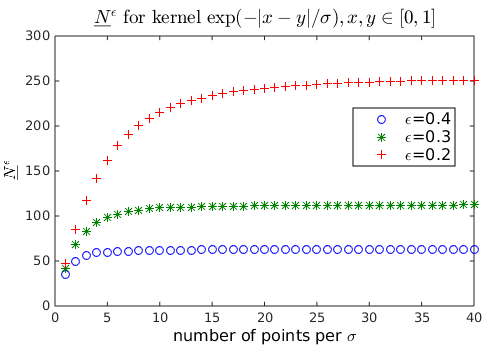}
			\caption*{(a) $\underline{N}^{\epsilon}$ vs. $r$ with fixed $\sigma=0.02$}
		\end{minipage}
		\hfil
		\begin{minipage}{0.46\textwidth}
			\includegraphics[width=1\textwidth]{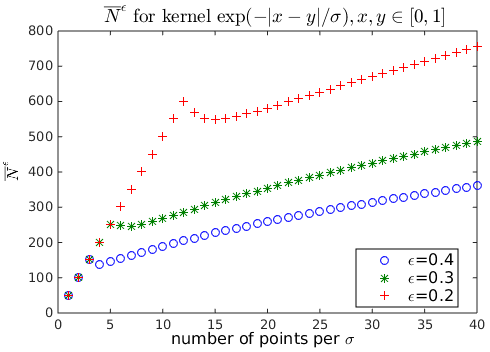}
			\caption*{(b) $\overline{N}^{\epsilon}$ vs. $r$ with fixed $\sigma=0.02$}
		\end{minipage}
		\\
		\vspace{0.3cm}
		\hspace{-0.6cm}
		\begin{minipage}{0.46\textwidth}
			\includegraphics[width=1\textwidth]{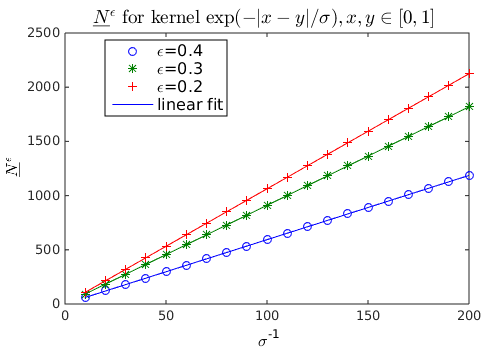}
			\caption*{(c) $\underline{N}^{\epsilon}$ vs. $\sigma$ with fixed $h=0.25\sigma$}
		\end{minipage}
		\hfil
		\begin{minipage}{0.46\textwidth}
			\includegraphics[width=1\textwidth]{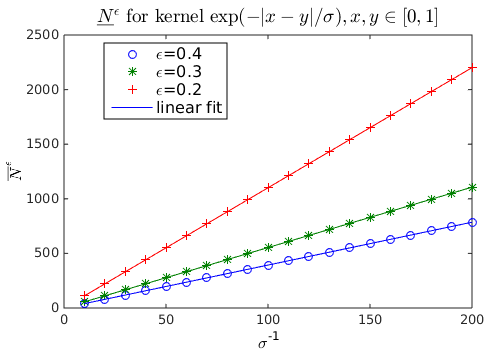}
			\caption*{(d) $\overline{N}^{\epsilon}$ vs $\sigma$ with fixed $h=0.25\sigma$}
		\end{minipage}
		\caption{Eigenvalue behavior for covariance function $\hat{C}_{\sigma}=\exp(-\frac{|x-y|}{\sigma})$ on the unit interval.}
		\label{fig:1D-ker2}
	\end{center}
\end{figure}



{\bf Example 2: unit square.} In this example, we show similar experiments for 2D random fields defined on the unit square with the two covariance functions defined in \eqref{eq:kernel}. Figure \ref{fig:square-ker1}  and Figure \ref{fig:square-ker2} shows the test results for the covariance function $\tilde{C}_{\sigma}$.  Figure \ref{fig:square-ker1} (a) shows how $\underline{N}^{\epsilon}$ as a function of the discretized resolution $r$, where $h=\sigma/r$, with a fixed $\sigma=0.1$. Similar to the 1D example above, a few points per $\sigma$ is fine enough to resolve the 2D covariance function with respect to the relative r.m.s error due to the normalization of  $\underline{N}^{\epsilon}$ and the smoothness of the covariance function $\tilde{C}_{\sigma}$. On the other hand, the unnormalized $\overline{N}^{\epsilon}$ grows slowly with the discretization resolution $r$. Figure \ref{fig:square-ker1} (c)\&(d) shows $\underline{N}^{\epsilon}$ and $\overline{N}^{\epsilon}$ as a function of $\sigma^{-1}$ for different $\epsilon$. We see a clear quadratic growth. The numerical results are based on the eigenvalue decomposition of the discretized covariance matrix for  $\tilde{C}_{\sigma}(x,y), x.y \in [0,1]\times[0,1]$ with $h=0.2\sigma$. 

\begin{figure}[!htb]
	\begin{center}
		\hspace{-0.6cm}
		\begin{minipage}{0.46\textwidth}
			\includegraphics[width=1\textwidth]{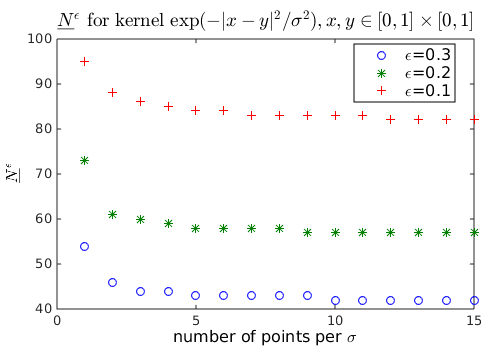}
			\caption*{(a) $\underline{N}^{\epsilon}$ vs. $r$ with fixed $\sigma=0.1$}
		\end{minipage}
		\hfil
		\begin{minipage}{0.46\textwidth}
			\includegraphics[width=1\textwidth]{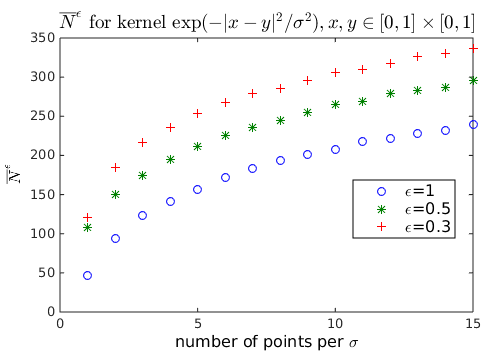}
			\caption*{(b) $\overline{N}^{\epsilon}$ vs. $r$ with fixed $\sigma=0.1$}
		\end{minipage}
		\\
		\vspace{0.3cm}
		\hspace{-0.6cm}
		\begin{minipage}{0.46\textwidth}
			\includegraphics[width=1\textwidth]{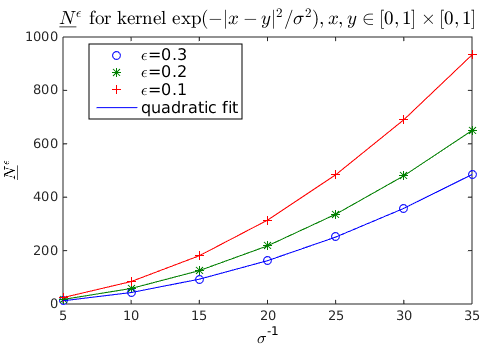}
			\caption*{(c) $\underline{N}^{\epsilon}$ vs. $\sigma$ with fixed $h=0.2\sigma$}
		\end{minipage}
		\hfil
		\begin{minipage}{0.46\textwidth}
			\includegraphics[width=1\textwidth]{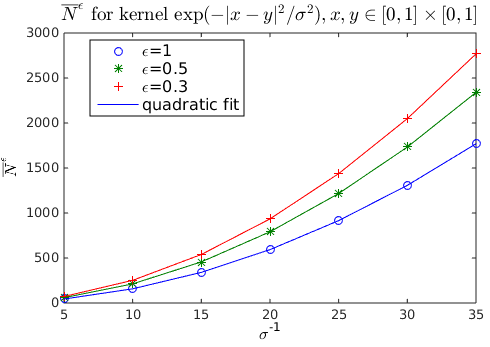}
			\caption*{(d) $\overline{N}^{\epsilon}$ vs. $\sigma$ with fixed $h=0.2\sigma$}
		\end{minipage}
		\caption{Eigenvalue behavior for covariance function $\tilde{C}_{\sigma}=\exp(-\frac{|x-y|^2}{\sigma^2})$ on the unit square.}
		\label{fig:square-ker1}
	\end{center}
\end{figure}

%

Figure \ref{fig:square-ker2} shows the corresponding results for covariance matrix $\hat{C}_{\sigma}$. Figure \ref{fig:square-ker2} (a) shows $\underline{N}^{\epsilon}$ as a function of the discretized resolution $r$ with a fixed $\sigma=0.1$. Due to the non-smoothness of $\hat{C}_{\sigma}(x,y)$ at $x=y$, $\underline{N}^{\epsilon}$ increases as the singularity is more resolved. However, it appears to be plateaued eventually.  While Figure \ref{fig:square-ker2}(b) shows that the unnormalized $\overline{N}^{\epsilon}$, the number of eigenvalues that are larger than $\epsilon$, grows with the discretization resolution $r$ and with a rate faster than that of the smooth covariance function. Figure \ref{fig:square-ker2} (c)\&(d) shows $\underline{N}^{\epsilon}$ and $\overline{N}^{\epsilon}$ as a function of $\sigma^{-1}$ for different $\epsilon$. Again, we see a clear quadratic growth. 

\begin{figure}[!htb]
	\begin{center}
		\hspace{-0.6cm}
		\begin{minipage}{0.46\textwidth}
			\includegraphics[width=1\textwidth]{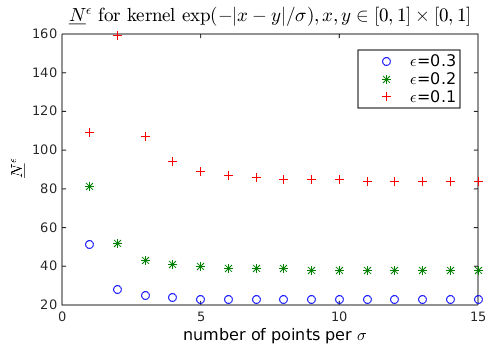}
			\caption*{(a) $\underline{N}^{\epsilon}$ vs. $r$ with fixed $\sigma=0.1$}
		\end{minipage}
		\hfil
		\begin{minipage}{0.46\textwidth}
			\includegraphics[width=1\textwidth]{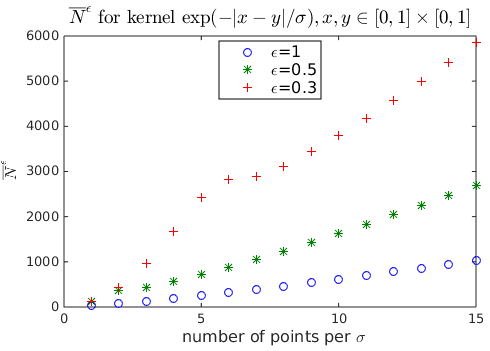}
			\caption*{(b) $\overline{N}^{\epsilon}$ vs. $r$ with fixed $\sigma=0.1$}
		\end{minipage}
		\\
		\vspace{0.3cm}
		\hspace{-0.6cm}
		\begin{minipage}{0.46\textwidth}
			\includegraphics[width=1\textwidth]{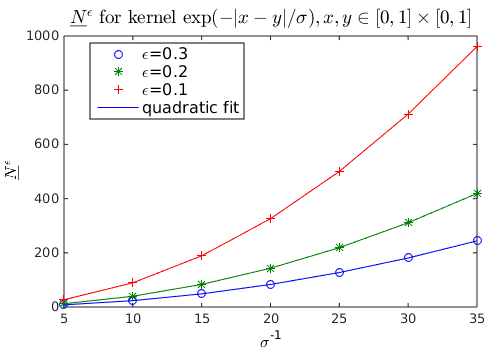}
			\caption*{(c) $\underline{N}^{\epsilon}$ vs. $\sigma$ with fixed $h=0.2\sigma$}
		\end{minipage}
		\hfil
		\begin{minipage}{0.46\textwidth}
			\includegraphics[width=1\textwidth]{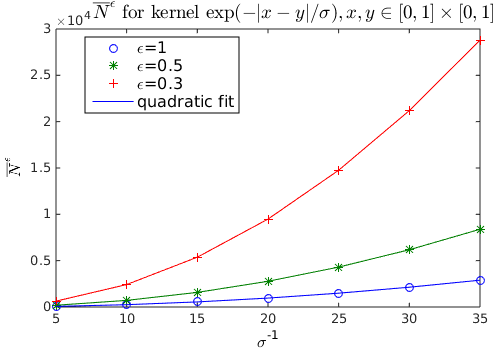}
			\caption*{(d) $\overline{N}^{\epsilon}$ vs. $\sigma$ with fixed $h=0.2\sigma$}
		\end{minipage}
		\caption{Eigenvalue behavior for covariance function $\hat{C}_{\sigma}=\exp(-\frac{|x-y|}{\sigma})$ on the unit square.}
		\label{fig:square-ker2}
	\end{center}
\end{figure}

%

{\bf Example: unit sphere.} In this example, we show results for random field defined on a unit sphere with the two covariance functions defined in \eqref{eq:kernel}, where $|x-y|$ is the geodesic distance between $x$ and $y$ on the unit sphere. The results are very similar to those on the unit square.

\begin{figure}[!htb]
	\begin{center}
		\hspace{-0.6cm}
		\begin{minipage}{0.46\textwidth}
			\includegraphics[width=1\textwidth]{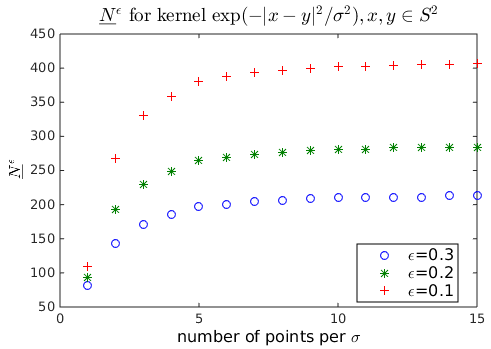}
			\caption*{(a) $\underline{N}^{\epsilon}$ vs. $r$ with fixed $\sigma=0.15$}
		\end{minipage}
		\hfil
		\begin{minipage}{0.46\textwidth}
			\includegraphics[width=1\textwidth]{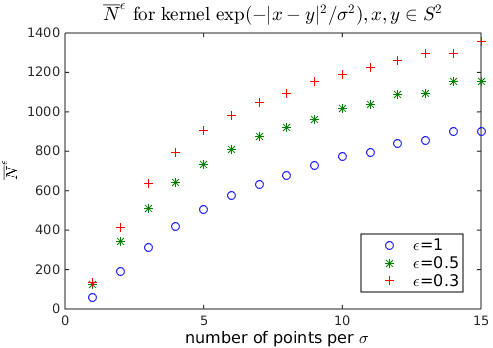}
			\caption*{(b) $\overline{N}^{\epsilon}$ vs. $r$ with fixed $\sigma=0.15$}
		\end{minipage}
		\\
		\vspace{0.3cm}
		\hspace{-0.6cm}
		\begin{minipage}{0.46\textwidth}
			\includegraphics[width=1\textwidth]{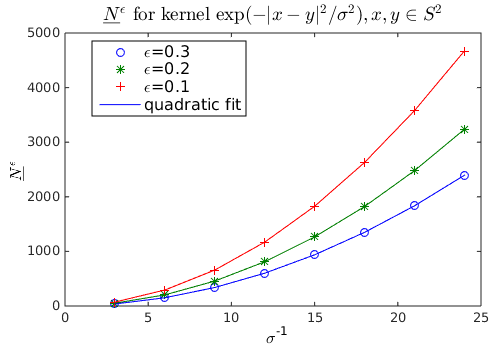}
			\caption*{(c) $\underline{N}^{\epsilon}$ vs. $\sigma$ with fixed $h=0.25\sigma$}
		\end{minipage}
		\hfil
		\begin{minipage}{0.46\textwidth}
			\includegraphics[width=1\textwidth]{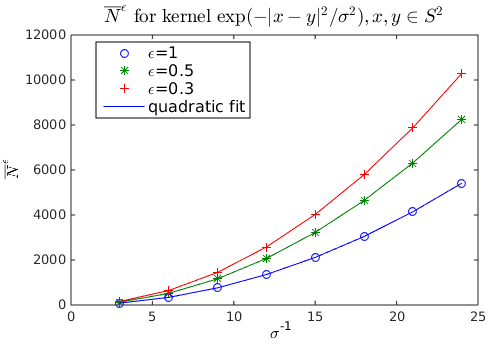}
			\caption*{(d) $\overline{N}^{\epsilon}$ vs $\sigma$ with fixed $h=0.25\sigma$}
		\end{minipage}
		\caption{Eigenvalue behavior for covariance function $\tilde{C}_{\sigma}=\exp(-\frac{|x-y|^2}{\sigma^2})$ on the unit sphere.}
		\label{fig:sphere-ker1}
	\end{center}
\end{figure}


\begin{figure}[!htb]
	\begin{center}
		\hspace{-0.6cm}
		\begin{minipage}{0.46\textwidth}
			\includegraphics[width=1\textwidth]{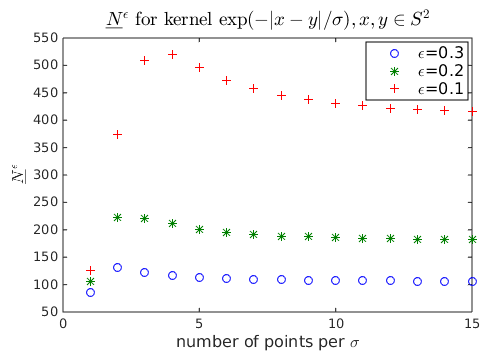}
			\caption*{(a) $\underline{N}^{\epsilon}$ vs. $r$ with fixed $\sigma=0.15$}
		\end{minipage}
		\hfil
		\begin{minipage}{0.46\textwidth}
			\includegraphics[width=1\textwidth]{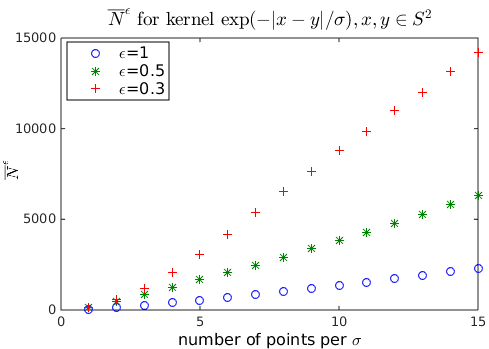}
			\caption*{(b) $\overline{N}^{\epsilon}$ vs. $r$ with fixed $\sigma=0.15$}
		\end{minipage}
		\\
		\vspace{0.3cm}
		\hspace{-0.6cm}
		\begin{minipage}{0.46\textwidth}
			\includegraphics[width=1\textwidth]{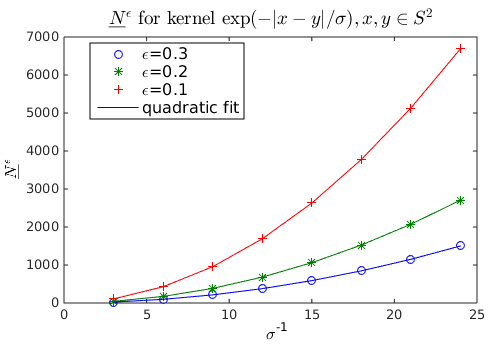}
			\caption*{(c) $\underline{N}^{\epsilon}$ vs. $\sigma$ with fixed $h=0.25\sigma$}
		\end{minipage}
		\hfil
		\begin{minipage}{0.46\textwidth}
			\includegraphics[width=1\textwidth]{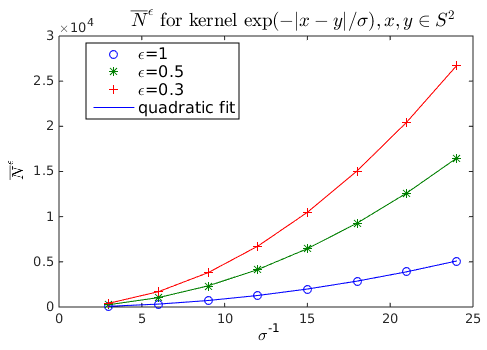}
			\caption*{(d) $\overline{N}^{\epsilon}$ vs $\sigma$ with fixed $h=0.25\sigma$}
		\end{minipage}
		\caption{Eigenvalue behavior for covariance function $\hat{C}_{\sigma}=\exp(-\frac{|x-y|}{\sigma})$ on the unit sphere.}
		\label{fig:sphere-ker2}
	\end{center}
\end{figure}

%


\section{Approximate Embedding of Random Vectors}
\label{sec:embedding}
Given a set of random vectors, $\bv_i \in \mathbb{R}^d, i =1, 2, \ldots, n$, a natural question is what is the dimension of the best linear space that can approximate this set of vectors to a certain error tolerance. For example, the set of random vectors can be the ensemble of random variables of a random field sampled at different locations. For exact embedding, lower bound and asymptotic lower bound are provided in terms of the covariance matrix $A=\{a_{ij}\}, a_{i,j}=\bv_i^T\bv_j, i,j=1, 2, \ldots, n$ \cite{alon2003problems}. Asymptotic upper bounds can be provided by Johnson-Lindenstrauss Lemma \cite{johnson1984extensions}. Here we introduce the definition of relative root mean square (r.m.s) $\epsilon$-embedding of a set of vectors. 

\begin{definition}
\label{def:eps}
A set of vectors $\{ \bv_i\}_{i=1}^n$ are $\epsilon$-embedded in a linear subspace, $S$, in relative r.m.s sense,  if 
$$\frac{\sum_{i=1}^n ||\bv_i-P_{S}\bv_i||_2^2}{\sum_{i=1}^n ||\bv_i||_2^2} \leq \epsilon^2$$ 
where $P_{S}\bv_i$ denotes the projection of $\bv_i$ in $S$.
\end{definition}
A natural question is what is the least dimension of all such linear subspaces defined as below.
\begin{definition}
\label{def:N}
For a set of vectors, define $\underline{N}^{\epsilon}$ to be the least dimension of all subspaces $S$ such that the vectors are relatively r.m.s. $\epsilon$-embedded in $S$.
\end{definition}

$\underline{N}^{\epsilon}$ is related to the principle component analysis (PCA) of a set of vectors or singular value decomposition (SVD) of the matrix $V = [\bv_1, \bv_2, ... \bv_n] \in \mathbb R^{d \times n}$, which is characterized by the eigen-decomposition of the covariance matrix $A= V^TV \in \mathbb R^{n \times n}$. Let $\lambda_1 \geq \lambda_2 \geq ... \geq \lambda_n \geq 0$ be the eigenvalues of $A$.  Then 
\begin{equation}
\sum_{i=1}^n ||\bv_i||_2^2= tr(A)  = \sum_{i=1}^n \lambda_i  .
\end{equation}
Moreover, the best linear subspace of dimension $l$ in $\mathbb{R}^d$, denoted by $\bar{S_l}$, that approximates the set of vectors $\{ v_i\}_{i=1}^n$ in the least squares sense is the space spanned by the first $l$ left singular vectors of $V$ and satisfies 
\begin{equation}\label{eq:SVD}
\sum_{i=1}^n ||\bv_i - P_{\bar{S_l}}\bv_i ||_2^2  = \sum_{i=l+1}^n \lambda_i
\end{equation}
Combining the above two equations, and the definitions of relatively r.m.s $\epsilon$-embedding \ref{def:eps} and $\underline{N}^{\epsilon}$ \ref{def:N}, we have  
\begin{equation}\label{eq:N}
\underline{N}^{\epsilon}=\min_{N} ~\mbox{such that} ~\frac{\sum_{i=N+1}^n \lambda_i}{\sum_{i=1}^n \lambda_i} \leq \epsilon^2 .
\end{equation} 
Now we give a general lower bound for $\underline{N}^{\epsilon}$. 
\begin{theorem}
\label{th:embedding}  
For a set of vectors $\bv_1, \bv_2, ... \bv_n \in \mathbb R^{d}$.  Let $\underline{N}^{\eps}$ be the least dimension of a linear subspace such that the set of vectors $\{ \bv_i\}_{i=1}^n$ can be relatively r.m.s $\eps$-embedded in that subspace.  Then
\[
\underline{N}^{\epsilon} \geq \frac{(\sum_{i=1}^n ||\bv_i||_2^2)^2 (1- \epsilon^2)^2}{\sum_{i,j=1}^n (\bv_i \cdot \bv_j)  ^2} = \frac{tr(A)^2 (1- \epsilon^2)^2}{||A||_F^2}, 
\]
where $V = [\bv_1, \bv_2, ... \bv_n] \in \mathbb R^{d \times n}$ and $A=V^TV$.
\end{theorem}
\begin{proof}
On one hand,  
\begin{equation}
||A||_F^2=tr(A^TA) = \sum_{i=1}^n \lambda_i^2  \geq \sum_{i=1}^{\underline{N}^\epsilon} \lambda_i^2  \geq \frac{1}{\underline{N}^{\epsilon}} \big( \sum_{i=1}^{\underline{N}^{\epsilon}} \lambda_i \big)^2.
\end{equation}
On the other hand, by the definition of $\underline{N}^{\epsilon}$, we have $\sum_{i=1}^{\underline{N}^{\epsilon}} \lambda_i  \ge  tr(A)(1- \epsilon^2)$. Combining the two we get the result.
\end{proof}

\subsection{Random vectors with i.i.d. entries}
\label{sec:iid}
In the case that the set of random vectors have i.i.d entries, one can derive a precise asymptotic formula for $\underline{N}^{\epsilon}$ using Mar$\check{\text{c}}$enko-Pastur law for random matrices. For completeness, we present Mar$\check{\text{c}}$enko-Pastur law first. 

\begin{theorem}[Mar$\check{\text{c}}$enko-Pastur \label{mp} \cite{marvcenko1967distribution}] Let $V$ be an $d \times n$ random matrix whose entries are i.i.d random variables with mean zero and variance $\sigma^2 < \infty$.  Let $\hat{A} = \frac{1}{d}V^TV$ and let $\hat{\lambda}_1 \geq... \geq \hat{\lambda}_n$ be the eigenvalues of $\hat{A}$.  Finally, consider the random measure $\mu_n(I) = \frac{1}{n} \# \{\hat{\lambda}_j \in I \} , \hspace{3mm} I \subset \mathbb R$.  Assume that $n,d \rightarrow \infty$ so that the ratio $n/d \rightarrow \alpha \in (0, + \infty)$.  Then $\mu_n \rightarrow \mu$ in weak* topology in distribution, where 
\[ \mu(I) =
  \begin{cases}
    (1-\frac{1}{\alpha}){\bf 1}_{0 \in I} + \nu(I),       & \quad \text{if }  \text{ $\alpha > 1$}\\
    \nu(I), & \quad \text{if } \text{$0 \leq \alpha \leq 1$}\\
  \end{cases}
\]

and
$$d\nu(x) = \frac{1}{2 \pi \sigma ^2} \frac{\sqrt{(\hat{\lambda}_+-x)(x-\hat{\lambda}_-)}}{\alpha x} {\bf 1}_{[\lambda_-, \lambda_+]} \hspace{1mm} dx$$

with
$$\hat{\lambda}_{\pm} = \sigma ^2 (1 \pm \sqrt{\alpha})^2.$$ \\
\end{theorem}

The following results for approximate embedding of a set of random vectors $\{\bv_i\}_{i=1}^n \in \mathbb{R}^d$ based on the Mar$\check{\text{c}}$enko-Pastur law assume i.i.d. entries of $\bv_i$ and are asymptotic as $n, d \rightarrow \infty$ and $\lim_{n\rightarrow \infty} \frac{d}{n}$ exists, i.e. for large $n$ (and $d$).  However, as will be shown by many numerical tests later, the formula is very accurate even for a single realization and quite small $n$ (and $d$).

\begin{theorem}
\label{th:asymptotics}  
For a set of random vectors $\{\bv_i\}_{i=1}^n \in \mathbb{R}^d$ whose entries are i.i.d with mean zero and variance = $\sigma^2 < \infty$. 
Let $V = [\bv_1, \bv_2, ... \bv_n] \in \mathbb R^{d\times n}$ and $\mu$ be the limit distribution of the eigenvalues of $\hat{A} = \frac{1}{d}V^TV$. 
Then the least dimension of a linear subspace, $\underline{N}^{\eps}$,  such that the vectors $\{ \bv_i\}_{i=1}^n$ can be relatively r.m.s. $\eps$-embedded in, has the following asymptotic formula:
\begin{equation}
\label{eq:Neps}
\frac{\underline{N}^{\epsilon}}{n} \rightarrow \int_{y}^{\hat{\lambda}_+} d\mu(x), ~ \mbox{as } n \rightarrow \infty, \quad \mbox{where } y ~\mbox{is such that } \int_{\hat{\lambda}_-}^{y} x d\mu(x) =  \sigma^2 \epsilon^2.
\end{equation}
\end{theorem}
\begin{proof}
Let $\hat{\lambda}_1 \geq \hat{\lambda}_2 \geq ... \geq \hat{\lambda}_n \geq 0$ be the eigenvalues of $\hat{A}$. Then
\[
\sum_{i=1}^n \hat{\lambda}_i  = tr(\hat{A}) = \sum_{i=1}^n \frac{1}{d}(\bv_i \cdot  \bv_i) \rightarrow n \sigma^2, \quad \mbox{as } n, d \rightarrow \infty.
\]
By definition, $\underline{N}^{\eps}$ is the smallest integer such  $\sum_{i=\underline{N}^{\epsilon}+1}^n \hat{\lambda}_i \leq  n \sigma^2 \epsilon^2 $.

Let $\mu$ be the limit measure for $\hat{A}$ as in the Mar$\check{\text{c}}$enko-Pastur law. We have that $\frac{1}{n} \sum_{i=\underline{N}^{\epsilon}+1}^n \hat{\lambda}_i \rightarrow \int_{\hat{\lambda}_-}^y x d\mu(x)$ and $\frac{\underline{N}^{\epsilon}}{n} \rightarrow \int_y^{\hat{\lambda}_+} d\mu(x)$.

\end{proof}

\begin{remark}
\label{re:rank}
Let $R^{\epsilon}$ be the the largest integer such that $\sqrt{\lambda_{R^{\epsilon}}}\ge \epsilon$ (the standard $\epsilon$ rank approximation of $V$). Under the same conditions in Theorem \ref{th:asymptotics} and use Mar$\check{\text{c}}$enko-Pastur law we have
\[
\frac{n-R^{\epsilon}+1}{n}\rightarrow \int_{\hat{\lambda}_-}^{\frac{\epsilon^2}{d}} d\mu(x) \quad \mbox{or} \quad \frac{R^{\epsilon}}{n}\rightarrow 1-\int_{\hat{\lambda}_-}^{\frac{\epsilon^2}{d}} d\mu(x) \quad \mbox{as } n\rightarrow \infty.
\]
\end{remark}

\begin{remark}
\label{re:derivative}
The asymptotic formula in Theorem \ref{th:asymptotics} says that, for any fixed tolerance $\epsilon>0$ in relative r.m.s sense, the dimension, $\underline{N}^{\epsilon}$, of the best linear subspace that can approximately embed a set of $n$ random vectors in $\mathbb{R}^d$ ($d=O(n)$) with i.i.d entries is proportional to $n$ for $n\gg 1$. Let's denote the ratio, $\rho(\epsilon)=\frac{\underline{N}^{\epsilon}}{n}$ as a function of $\epsilon$, 
one can compute the rate of change of $\rho(\epsilon)$ with respect to $\epsilon$. From \eqref{eq:Neps} we have
\begin{equation}
\frac{d\rho(\epsilon)}{d\epsilon}=- \frac{\sqrt{(\hat{\lambda}_+ - y)(y- \hat{\lambda}_-)}}{2  \pi \sigma^2 \alpha y}  \frac{dy}{d \epsilon}=-\frac{2 \sigma^2 \epsilon}{y} , \quad \epsilon \in (0, 1),
\end{equation}
where relation $\int_{\hat{\lambda}_-}^{y} x d\mu(x) =  \sigma^2 \epsilon^2$ is used for the last equality. Let $\epsilon\rightarrow 0+$, which implies $y\rightarrow \hat{\lambda}_-$, then 
\[
\frac{d \rho(\epsilon)}{d\epsilon} = \left\{
\begin{array}{ll}
O( \epsilon), & \mbox{if } \lim_{\rightarrow\infty}\frac{n}{d}=\alpha\ne 1
\\
O(\epsilon^{-\frac{1}{3}}) &  \mbox{if } \lim_{n\rightarrow\infty}\frac{n}{d}=\alpha = 1
\end{array}\right.
\quad \mbox{as } \epsilon \rightarrow 0^+
\]
For $\alpha=1$, we use the fact that 
\[
\frac{dy}{d \epsilon} = \frac{4 \pi \sigma^4  \epsilon}{\sqrt{y(\hat{\lambda}_+-y)} }.
\]
Figure \ref{fig:iid} (d) shows numerical plots of $\rho(\epsilon)$ for different $\alpha=\frac{n}{d}$, and we see the singular behavior of the slope of $\rho(\epsilon)$ for $\alpha=1$.
\end{remark}

We can also ask the dual question: if one projects a set of vectors onto the best $k$-dimensional linear subspace, at least how much "error" in the relative r.m.s sense does one have to make? 

\begin{corollary}\label{co:cor1}   Given vectors $\{ \bv_i\}_{i=1}^n \in \mathbb R^d$ and a fixed $k$, $k<d$.  The relative r.m.s. error for the best $k$ dimensional linear subspace is asymptotically given by $\sqrt{\frac{1}{\sigma^2} \int_{\hat{\lambda}_-}^y x d\mu(x) }$, where $y$ satisfies $\int_{\hat{\lambda}_-}^y d\mu(x) = \frac{n-k}{n}$. 
\end{corollary}
\begin{proof}
Let $y = \hat{\lambda}_{k+1}$ be the value of the $(k+1)^\text{th}$ largest eigenvalue of $\hat{A} = \frac{1}{d}V^TV$, where $V = [\bv_1, \bv_2, ... \bv_n] \in \mathbb R^{d\times n}$.  By the Mar$\check{\text{c}}$enko-Pastur law, $\frac{n-k}{n}\rightarrow \int_{\hat{\lambda}_-}^y d\mu(x) $ and 
$\frac{1}{n} \sum_{i=k+1}^n \hat{\lambda}_i \rightarrow \int_{\hat{\lambda}_-}^{y} x d\mu(x)$. Since $\frac{1}{n} \sum_{i=1}^n \hat{\lambda}_i \rightarrow \sigma^2$, we have 
\[
\frac{ \sum_{i=k+1}^n \lambda_i}{ \sum_{i=1}^n \lambda_i}=\frac{ \sum_{i=k+1}^n \hat{\lambda}_i}{ \sum_{i=1}^n \hat{\lambda}_i}=
\frac{1}{\sigma^2} \int_{\hat{\lambda}_-}^y x d\mu(x) 
\]

\end{proof}

\begin{remark}
In practice,  even if one does not have the knowledge of the probability distribution for a set of vectors, as long as the entries are approximately i.i.d, one can compute the empirical mean and variance from the data and use them to estimate $\underline{N}^{\epsilon}$ from the above formulas.
\end{remark}

\subsection{Random vectors with a covariance structure}
\label{sec:correlated}
In this section, we study the scaling law of $\underline{N}^{\epsilon}$ for a set of random vectors $\bv_i \in \mathbb{R}^d, i=1, 2, \ldots, n$,
$[\bv_1, \bv_2,\dots, \bv_n]=V=XL^T$, where $X$ is a random matrix whose entries are i.i.d with mean $0$ and variance $1$, and $L$ is the Cholesky decomposition of a covariance matrix $C=LL^T$.
 If the eigenvalues 
of the of the covariance matrix $C$ have a limit distribution as $n \rightarrow \infty$,  one can find the empirical distribution of the eigenvalues of $V^TV$ using the generalized Mar$\check{\text{c}}$enko-Pastur law \cite{silverstein1995empirical, silverstein1995strong} through an integral equation that relates the Stieltjes transforms of the two limit distributions, which will provide the explicit asymptotic scaling law for $\frac{\underline{N}^{\epsilon}}{n}$ as in Theorem \ref{th:asymptotics}
Below we will show a complete calculation for the eigenvalues of covariance matrix of the following form: $C_{i,j}=cov(\bv_i,\bv_j)=\exp(-\frac{|i-j|}{\sigma})$ and give the explicit asymptotic scaling law for $\frac{\underline{N}^{\epsilon}}{n}$ for the covariance matrix. 


Suppose $(\lambda_k, \be_k), k=1, 2, \dots n$ are the ordered eigenvalues and eigenvectors for the matrix $C_{i,j}=\exp(-\frac{|i-j|}{\sigma})$. Let $\tau=\frac{1}{\sigma}$ and $\be_k=[e_1^k, e_2^k, \ldots, e_{n}^k]^T$, then 
\begin{equation}
\sum_{j=1}^n \exp(-\tau|i-j|)e_j^k=\lambda_k e_i^k ~~ \Rightarrow ~~ e^k_{j+1}-2(\cosh \tau-\frac{\sinh\tau}{\lambda_k} )e^k_j+e^k_{j-1}=0,  ~1<j<n.
\end{equation}
Such recursive relation implies that $e^k_j$ can be represented by (with scaling)
	\begin{equation}
	e^k_j= \cos(j \theta_k + \psi_k)
	\end{equation}
for some $\theta_k$ and $\psi_k$, where $\theta_k\in(0, \pi)$ satisfies 
	\begin{equation}\label{eq:k}
	\cos\theta_k = \cosh \tau - \frac{1}{\lambda_k}\sinh\tau
	\end{equation}
On the other hand, for the top and bottom element we get the following two boundary conditions:
	\begin{equation}\label{eq:tau}
	\begin{aligned}
	(\exp\tau - \frac{2}{\lambda_k}\sinh\tau) \cos(\psi_k + \theta_k) &= \cos(\psi_k + 2\theta_k),\\
	(\exp\tau - \frac{2}{\lambda_k}\sinh\tau) \cos(\psi_k + n\theta_k) &= \cos(\psi_k + (n-1)\theta_k).
	\end{aligned}
	\end{equation}
Use equation~\eqref{eq:k} in~\eqref{eq:tau}, we obtain the relation between $\psi_k$ and $\theta_k$,
	\begin{equation}\label{eq:psi}
	\tan(\psi_k) = \frac{\cos\theta_k - \exp\tau }{\sin\theta_k}.
	\end{equation}
Also~\eqref{eq:tau} and $\eqref{eq:psi}$ imply that $\psi_k \in (-\frac{\pi}{2}, 0)$ and 
\begin{equation}
\label{eq:relation}
2\psi_k + (n+1)\theta_k = (k-1)\pi,~~~1\le k \le n
\end{equation}
 More specifically, we have the following equations for $\theta_k$, 
	\begin{equation}
	\label{eq:theta}
(n+1)\theta_k + 2\arctan\left(\frac{\cos\theta_k - \exp\tau}{\sin\theta_k}\right) = (k-1)\pi, ~~ \theta_k \in \left(\frac{(k-1)\pi}{(n+1)}, \frac{k\pi}{(n+1)}\right).
	\end{equation}
Actually, one can derive a better range for $\theta_k$. Let us denote $\eta_k = (n+1)\theta_k - (k-1)\pi$. Equation \eqref{eq:relation} can be transformed to
	\begin{equation}
	\tan\left(\frac{\eta_k}{2}\right) = \tan\left(\frac{\theta_k}{2}\right) + \frac{\exp(\tau) - 1}{\sin\theta_k}> \tan\left(\frac{\theta_k}{2}\right) \quad \eta_k > \theta_k.
	\end{equation}
Hence we have the following distribution for $\theta_k$
\begin{equation}
 \theta_k \in (\frac{(k-1)\pi}{n}, \frac{k\pi}{n+1}).
\end{equation}
Once $\theta_k$ is solved from \eqref{eq:theta}, one can compute the corresponding eigenvalues:
\begin{equation}
\label{eq:e-value}
\lambda_k = \frac{\sinh\tau}{\cosh\tau - \cos(\theta_k)}.
\end{equation}

From the formula \eqref{eq:e-value} for $\lambda_k$ and the distribution \eqref{eq:theta} for $\theta_k$, we can derive the exact asymptotic formula for $N^{\epsilon}/n$. Since 
\[
\frac{\sinh\tau}{\cosh \tau - \cos(\frac{k-1}{n}\pi)}\le \lambda_k \le \frac{\sinh\tau}{\cosh \tau - \cos(\frac{k}{n+1}\pi)},
\]

we have, as $n \rightarrow \infty$,
\begin{equation}
\frac{1}{n}\sum_{k=1}^{n} \lambda_k\rightarrow \frac{1}{\pi} \int_{0}^{\pi}\frac{\sinh \tau}{\cosh\tau - \cos x} dx 
\end{equation}
and 
\begin{equation}
\frac{1}{n}\sum_{k=\underline{N}^{\epsilon}}^{n} \lambda_k\rightarrow \frac{1}{\pi} \int_{\frac{\pi\underline{N}^{\epsilon}}{n}}^{\pi}\frac{\sinh \tau}{\cosh\tau - \cos x} dx.
\end{equation}
Since $\underline{N}^{\epsilon}$ satisfies 
	\begin{equation}
	\sum_{k=\underline{N}^{\epsilon}+1}^{n} \lambda_k	=  \epsilon^2\left(\sum_{k=1}^{n} \lambda_k \right),
	\end{equation}
	it is to find out $t = \frac{\underline{N}^{\epsilon}}{n}$ such that
	\begin{equation}
	\int_{t\pi}^{\pi}\frac{\sinh \tau}{\cosh\tau - \cos x} dx = \epsilon^2 	\int_{0}^{\pi}\frac{\sinh \tau}{\cosh\tau - \cos x} dx.
	\end{equation}
	Then we solve $t$ and get
	\begin{equation}
	\label{eq:t}
	 \frac{\underline{N}^{\epsilon}}{n} \rightarrow t = \frac{2}{\pi} \arctan\left(\tanh\left(\frac{\tau}{2}\right) \tan\left(\frac{\pi}{2}(1 - \epsilon^2)\right)\right) \quad \mbox{as } n\rightarrow \infty.
	\end{equation}

\subsection{Experiments} 
We present a few numerical experiments to verify our asymptotic formula for $\underline{N}^{\epsilon}$ derived in the previous sections for different scenarios. In all our tests, numerical results are computed from random vectors generated by one realization. No ensemble average is performed.

{\bf Example: random vectors with i.i.d entries.} Figure \ref{fig:iid}(a), (b) plots the asymptotic formula of $\underline{N}^{\epsilon}$ in terms of the total number of random vectors $n$ (solid line) vs. the numerically computed $\underline{N}^{\epsilon}$ for random vectors with i.i.d. Gaussian entries and i.i.d Bernoulli respectively, all with mean 0 and variance 1. In our tests, we set $n=\frac{d}{4}$ and show results for different $\epsilon$. We see remarkable agreements between our asymptotic estimate in Theorem \ref{th:asymptotics} and the numerical result even for a quite small number of random vectors in one realization. Figure \ref{fig:iid}(c) plots the asymptotic formula of $R^{\epsilon}$ ($\epsilon$ rank) in Remark \ref{re:rank} (solid line) vs. the numerically computed results for different $\epsilon$ for random vectors with i.i.d. Gaussian entries with mean 0 and variance 1. Figure \ref{fig:iid}(d) plots the ratio $\rho(\epsilon)=\frac{\underline{N}^{\epsilon}}{n}$ as a function of $\epsilon$ for random vectors with i.i.d. Gaussian entries with mean 0 and variance 1. We can see the singularity of $\left.\frac{d\rho(\epsilon)}{d\epsilon}\right|_{\epsilon=0}$ when $\frac{n}{d}=1$.
\begin{figure}[htb]
\begin{center}
\begin{tabular}{cc}
\includegraphics[width=0.55\hsize]{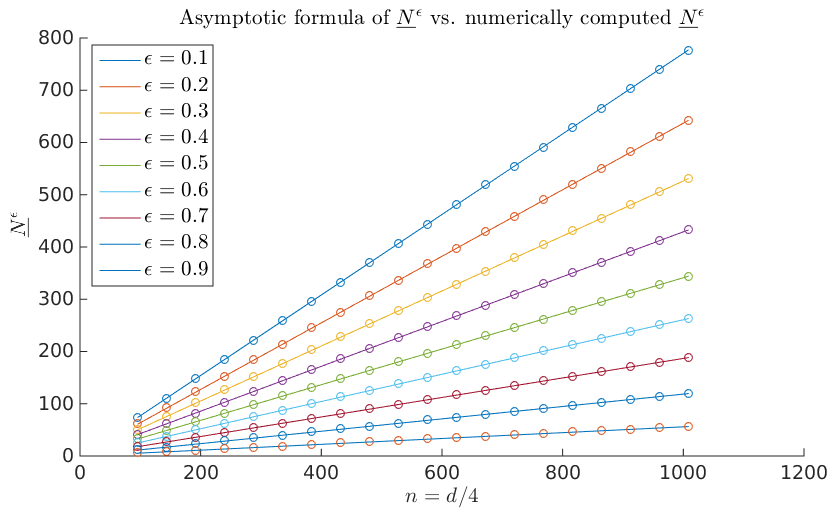} & \includegraphics[width=0.55\hsize]{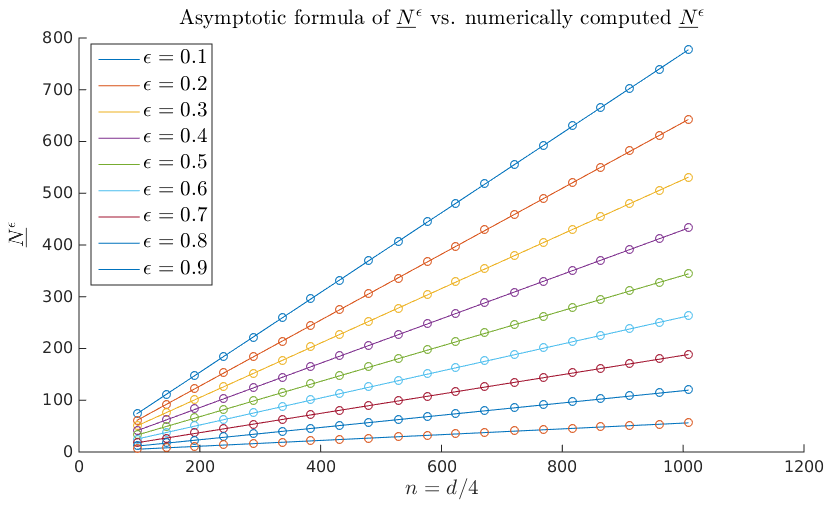}
\\
(a) random vectors with Gaussian i.i.d entries & (b) random vectors with Bernoulli i.i.d entries
\\
\includegraphics[width=0.55\hsize]{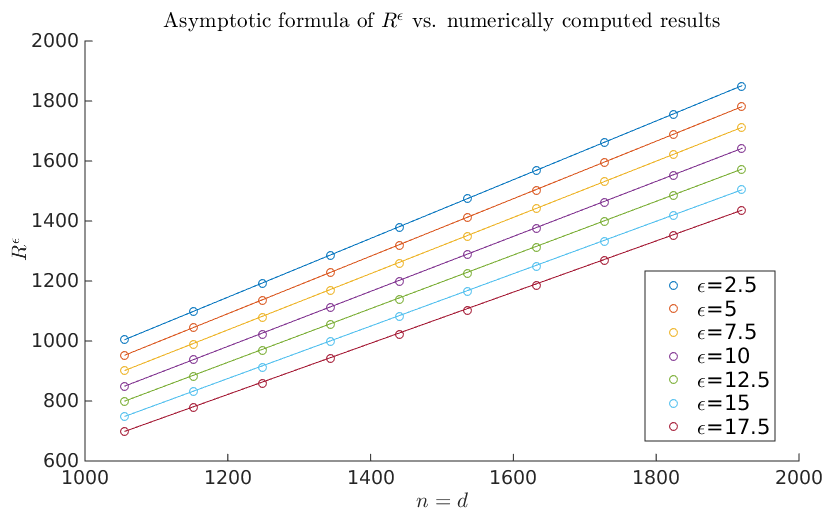}  & \includegraphics[width=0.55\hsize]{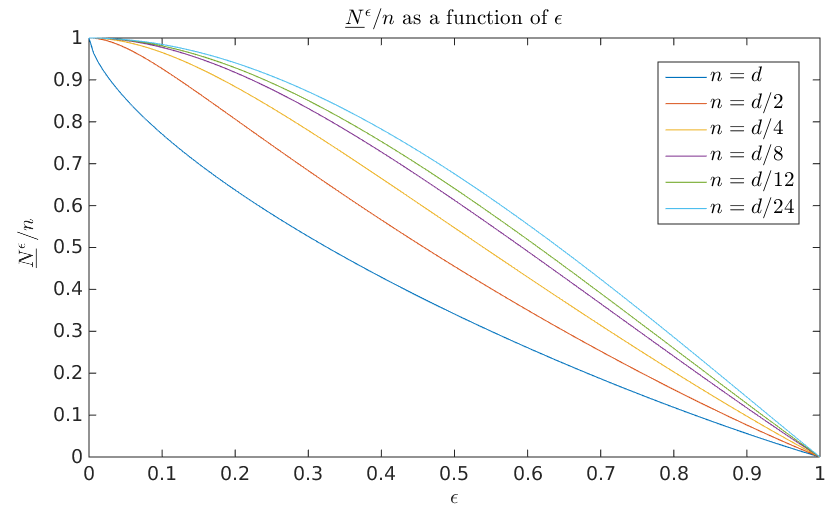} 
\\
(c) $R^{\epsilon}$ & (d) $\rho(\epsilon)$
\end{tabular}
\end{center}
\caption{}
\label{fig:iid}
\end{figure}

{\bf Example: random vectors with given covariance.} In this test, we first generate random vectors with covariance matrix $C_{i,j}=\exp(-\frac{|i-j|}{\sigma})$ by $V=XL^T$ where $X$ is a Gaussian matrix and $C=LL^T$. Figure \ref{fig:correlated} (a) plots  $\underline{N}^{\epsilon}$ computed from our asymptotic formula \eqref{eq:t} vs. numerically computed one from $C$. Figure \ref{fig:correlated} (b) plots $\underline{N}^{\epsilon}$ vs. $\frac{1}{\sigma}$ for $C$.  As we can see from the plot, there is a linear scaling regime for $\underline{N}^{\epsilon}$ vs. $\frac{1}{\sigma}$ when $\sigma$ is large compared to 1. In this regime, the behavior is analogous to the one of a 1D random field with a similar covariance function and correlation length $\sigma$ for which the intrinsic complexity or degrees of freedom is $O(\frac{1}{\sigma})$. However, when $\sigma$ gets close to 1 and smaller, the entries of random matrix $V$ becomes almost i.i.d. Hence $\frac{\underline{N}^{\epsilon}}{n}$ should converges to 
the asymptotic estimate in Theorem \ref{th:asymptotics} as $\sigma\rightarrow 0$. This is shown in Figure \ref{fig:convergence}.  Figure \ref{fig:correlated} (c), (d) shows similar numerical tests for random vectors with covariance matrix $C_{i,j}=\exp(-\frac{|i-j|^2}{\sigma^2})$, although we do not have a analytical solution to compare. However, the numerical evidence suggests that $\rho(\epsilon)=\lim_{n\rightarrow \infty}\frac{\uNe}{n}$  does exist. From the relation stated in Theorem \ref{th:asymptotics}, one may deduce the limit distribution of the eigenvalues. 

\begin{remark}
We design the following iterative method to solve ~\eqref{eq:theta} for $\theta_k$ first. Let $\theta^0_k = \frac{(k-1)\pi}{n}$, then solving the following equation a few times will achieve sufficient accuracy,
	\begin{equation}
	(n+1)\theta^{n}_k + 2\arctan\left(\frac{\cos\theta^{n-1}_k - \exp\tau}{\sin\theta^{n-1}_k}\right) = (k-1)\pi,~~~1\le k\le n
	\end{equation}
	Then eigenvalue $\lambda_k$ is computed through
\[
	\lambda_k = \frac{\sinh\tau}{\cosh\tau - \cos(\theta_k)}.
\]
Actually our iterative method is much faster than using $\text{eig}(A)$ from $\text{MATLAB}$. 
\end{remark}

\begin{figure}[htb]
\begin{center}
\begin{tabular}{cc}
\multicolumn{2}{c}{for covariance matrix $C_{i,j}=\exp(-\frac{|i-j|}{\sigma})$}
\\
\includegraphics[width=0.5\hsize]{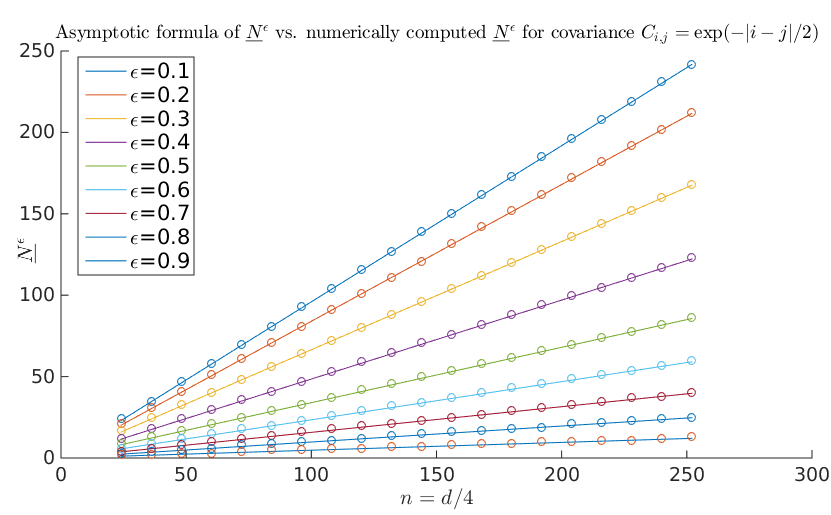} & \includegraphics[width=0.5\hsize]{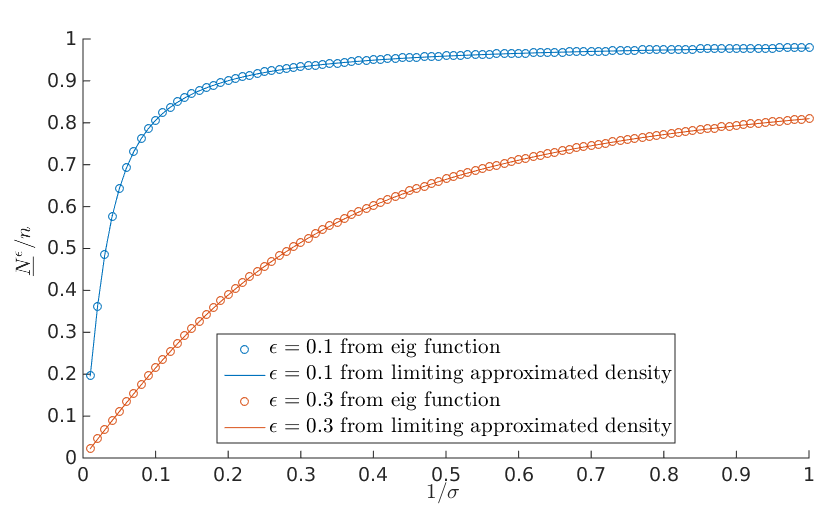}
\\
(a)   & (b) 
\\
\multicolumn{2}{c}{for covariance matrix $C_{i,j}=\exp(-\frac{(i-j)^2}{2\sigma^2})$}
\\
\includegraphics[width=0.5\hsize]{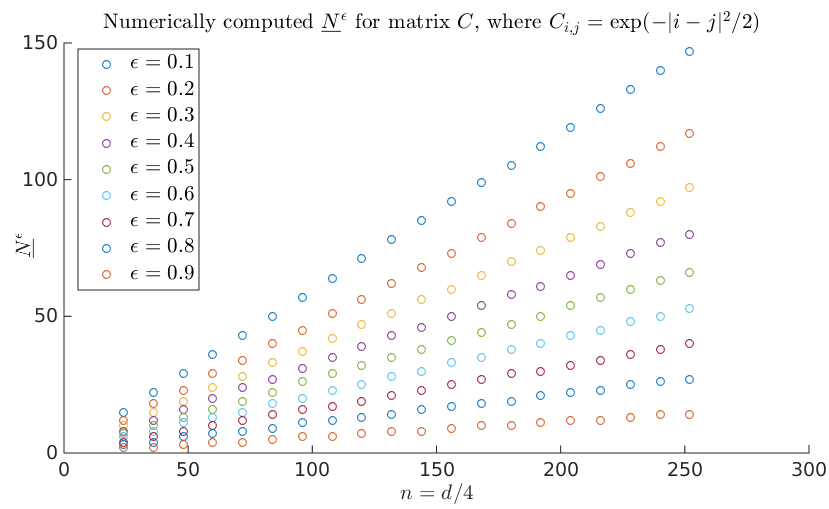} & \includegraphics[width=0.5\hsize]{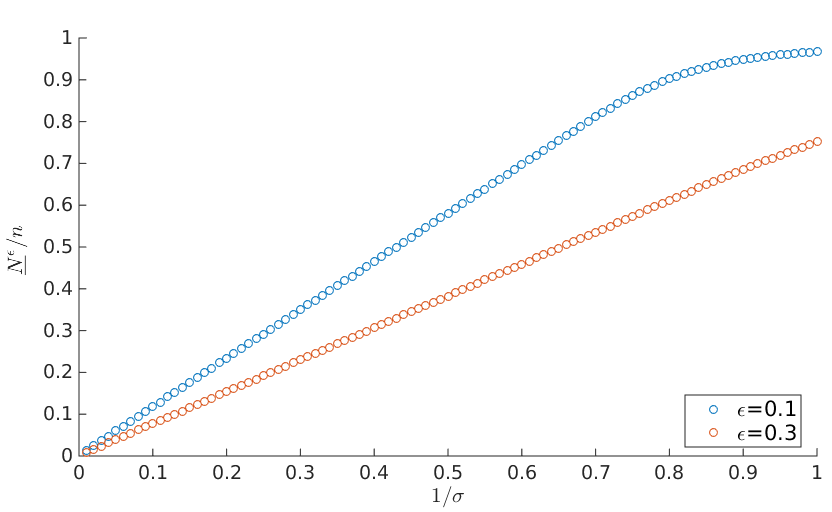}
\\
(c)   & (d) 
\end{tabular}
\end{center}
\caption{}
\label{fig:correlated}
\end{figure}
	
\begin{figure}[htb]
\includegraphics[width=0.5\hsize]{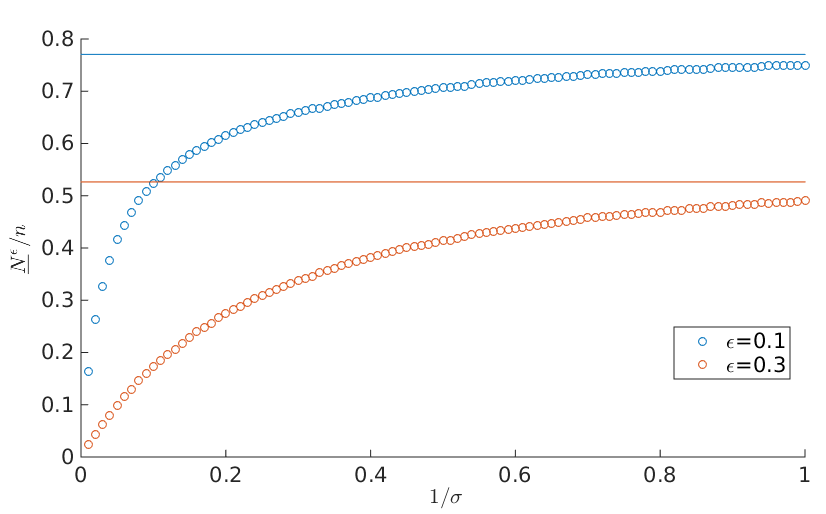}	
\caption{}
\label{fig:convergence}
\end{figure}

{\bf Acknowledgement} H. Zhao would like to thank Andrew Stuart for referring the work \cite{schwab2006karhunen} and helpful discussions. The research of H. Zhao is partially supported by NSF grant DMS-1418422 and DMS-1622490. The research of J. Bryson is partially supported by NSF Graduate Research Fellowship Program DGE-1321846.

\bibliographystyle{siam}
\bibliography{scaling}

\begin{thebibliography}{10}

\bibitem{alon2003problems}
{\sc N.~Alon}, {\em Problems and results in extremal combinatorics—i},
  Discrete Mathematics, 273 (2003), pp.~31--53.

\bibitem{ambikasaran2014fast}
{\sc S.~Ambikasaran, D.~Foreman-Mackey, L.~Greengard, D.~W. Hogg, and
  M.~O'Neil}, {\em Fast direct methods for gaussian processes}, arXiv preprint
  arXiv:1403.6015,  (2014).

\bibitem{bebendorf2003existence}
{\sc M.~Bebendorf and W.~Hackbusch}, {\em Existence of $\mathcal{H}$-matrix
  approximants to the inverse fe-matrix of elliptic operators with
  $l^{\infty}$-coefficients}, Numerische Mathematik, 95 (2003), pp.~1--28.

\bibitem{bochner1948several}
{\sc S.~Bochner and W.~T. Martin}, {\em Several complex variables},  (1948).

\bibitem{borm2010approximation}
{\sc S.~B{\"o}rm}, {\em Approximation of solution operators of elliptic partial
  differential equations by $\mathcal{H}$-and $\mathcal{H}^2$-matrices},
  Numerische Mathematik, 115 (2010), pp.~165--193.

\bibitem{brenner2007mathematical}
{\sc S.~Brenner and R.~Scott}, {\em The mathematical theory of finite element
  methods}, vol.~15, Springer Science \& Business Media, 2007.

\bibitem{candes2009fast}
{\sc E.~Candes, L.~Demanet, and L.~Ying}, {\em A fast butterfly algorithm for
  the computation of fourier integral operators}, Multiscale Modeling \&
  Simulation, 7 (2009), pp.~1727--1750.

\bibitem{engquist11approximate}
{\sc B.~Engquist and H.~Zhao}, {\em Approximate separability of the green’s
  functions of the helmholtz equation in the high frequency limit},
  Communications on Pure and Applied Mathematics, 11, pp.~1--0602.

\bibitem{greengard1988rapid}
{\sc L.~Greengard}, {\em The rapid evaluation of potential fields in particle
  systems}, MIT press, 1988.

\bibitem{greengard1987fast}
{\sc L.~Greengard and V.~Rokhlin}, {\em A fast algorithm for particle
  simulations}, Journal of computational physics, 73 (1987), pp.~325--348.

\bibitem{ho2016hierarchical}
{\sc K.~L. Ho and L.~Ying}, {\em Hierarchical interpolative factorization for
  elliptic operators: differential equations}, Communications on Pure and
  Applied Mathematics, 69 (2016), pp.~1415--1451.

\bibitem{johnson1984extensions}
{\sc W.~B. Johnson and J.~Lindenstrauss}, {\em Extensions of lipschitz mappings
  into a hilbert space}, Contemporary mathematics, 26 (1984), p.~1.

\bibitem{kolmogoroff1936uber}
{\sc A.~Kolmogoroff}, {\em Uber die beste annaherung von funktionen einer
  gegebenen funktionenklasse}, Annals of Mathematics,  (1936), pp.~107--110.

\bibitem{lehmann2006testing}
{\sc E.~L. Lehmann and J.~P. Romano}, {\em Testing statistical hypotheses},
  Springer Science \& Business Media, 2006.

\bibitem{marvcenko1967distribution}
{\sc V.~A. Mar{\v{c}}enko and L.~A. Pastur}, {\em Distribution of eigenvalues
  for some sets of random matrices}, Mathematics of the USSR-Sbornik, 1 (1967),
  p.~457.

\bibitem{martinsson2009fast}
{\sc P.-G. Martinsson}, {\em A fast direct solver for a class of elliptic
  partial differential equations}, Journal of Scientific Computing, 38 (2009),
  pp.~316--330.

\bibitem{michielssen1996multilevel}
{\sc E.~Michielssen and A.~Boag}, {\em A multilevel matrix decomposition
  algorithm for analyzing scattering from large structures}, IEEE Transactions
  on Antennas and Propagation, 44 (1996), pp.~1086--1093.

\bibitem{rokhlin1990rapid}
{\sc V.~Rokhlin}, {\em Rapid solution of integral equations of scattering
  theory in two dimensions}, Journal of Computational Physics, 86 (1990),
  pp.~414--439.

\bibitem{schmitz2012fast}
{\sc P.~G. Schmitz and L.~Ying}, {\em A fast direct solver for elliptic
  problems on general meshes in 2d}, Journal of Computational Physics, 231
  (2012), pp.~1314--1338.

\bibitem{schwab2006karhunen}
{\sc C.~Schwab and R.~A. Todor}, {\em Karhunen--lo{\`e}ve approximation of
  random fields by generalized fast multipole methods}, Journal of
  Computational Physics, 217 (2006), pp.~100--122.

\bibitem{silverstein1995strong}
{\sc J.~W. Silverstein}, {\em Strong convergence of the empirical distribution
  of eigenvalues of large dimensional random matrices}, Journal of Multivariate
  Analysis, 55 (1995), pp.~331--339.

\bibitem{silverstein1995empirical}
{\sc J.~W. Silverstein and Z.~Bai}, {\em On the empirical distribution of
  eigenvalues of a class of large dimensional random matrices}, Journal of
  Multivariate analysis, 54 (1995), pp.~175--192.

\bibitem{trefethen2017multivariate}
{\sc L.~Trefethen}, {\em Multivariate polynomial approximation in the
  hypercube}, Proceedings of the American Mathematical Society, 145 (2017),
  pp.~4837--4844.

\bibitem{udell2017nice}
{\sc M.~Udell and A.~Townsend}, {\em Nice latent variable models have
  log-rank}, arXiv preprint arXiv:1705.07474,  (2017).

\bibitem{xia2013efficient}
{\sc J.~Xia}, {\em Efficient structured multifrontal factorization for general
  large sparse matrices}, SIAM Journal on Scientific Computing, 35 (2013),
  pp.~A832--A860.

\end{thebibliography}

\end{document}